\documentclass[11pt]{amsart}
\usepackage[T1]{fontenc}
\usepackage[latin9]{inputenc}
\usepackage{geometry}
\usepackage{comment}
\usepackage{bm}
\usepackage{dsfont}
\usepackage{enumerate}
\usepackage{amsmath}
\usepackage{amsthm}
\usepackage{amssymb}
\usepackage{color}
\usepackage{accents}
\usepackage{hyperref} 
\makeatletter

\theoremstyle{plain}
\newtheorem{thm}{\protect\theoremname}[section]
  \theoremstyle{definition}
  \newtheorem{defn}[thm]{\protect\definitionname}
      \theoremstyle{plain}
  \newtheorem{prop}[thm]{\protect\propositionname}
    \newtheorem{coro}[thm]{Corollary}
    \theoremstyle{thm}
  \newtheorem{lem}[thm]{\protect\lemmaname}
  \theoremstyle{remark}
  \newtheorem{rem}[thm]{\protect\remarkname}
  \theoremstyle{example}
  \newtheorem{example}[thm]{\protect\examplename}

\makeatother

\usepackage{babel}
  \providecommand{\definitionname}{Definition}
  \providecommand{\examplename}{Example}
  \providecommand{\remarkname}{Remark}
\providecommand{\theoremname}{Theorem}
\providecommand{\lemmaname}{Lemma}
\providecommand{\propositionname}{Proposition}

\newcommand{\abs}[1]{\ensuremath{\left|\text{}#1\text{}\right|}}

\newcommand{\mytilde}[1]{#1'}

\newcommand{\vx}{\mathbf x}
\newcommand{\vy}{\mathbf y}

\newcommand{\va}{\mathbf a}

\newcommand{\cI}{\mathcal I}
\newcommand{\dist}{{\rm dist}}
\newcommand{\conv}{{\rm conv}}
\newcommand{\ri}{\mathrm i}

\newcommand{\mycircledast}{\ast}
\newcommand{\conj}[1]{\overline{#1}}

\newcommand{\R}{\mathbb{R}}

\newcommand{\Z}{\mathbb{Z}}

\newcommand{\vk}{\mathbf{k}}

\newcommand{\innerp}[1]{\langle {#1} \rangle}

\newcommand{\cl}{\mathfrak{cl}\,}

\begin{document}

%\title{Determining Indicator Functions on $\mathbb{R}$ from Fourier Transform Magnitudes }
\title{Fourier Phase Retrieval for Finite Unions of Intervals }

\author{Yu Xia}
\thanks{ Yu Xia was supported by NSFC grant (12271133, U21A20426).}
\address{School of Mathematics, Hangzhou Normal University, Hangzhou 311121, China}
\email{yxia@hznu.edu.cn}

\author{Zhiqiang Xu}
\thanks{ Zhiqiang Xu is supported  by NSFC (12471361, 12021001, 12288201) . }
\address{ State Key Laboratory of Mathematical Sciences, Academy of Mathematics and Systems Science, Chinese Academy of Sciences, Beijing 100190, China\newline   School of Mathematical Sciences, University of Chinese Academy of Sciences, Beijing 100049, China.}
\email{xuzq@lsec.cc.ac.cn}
\date{}
\maketitle

\begin{abstract}

This paper investigates the one-dimensional Fourier phase retrieval problem for indicator functions of finite unions of intervals. Specifically, we study the recovery of a set $\Omega = \bigcup_{j=1}^m I_j \subset \mathbb{R}$ from the magnitude of its Fourier transform $|\widehat{\mathbf{1}_\Omega}|$, where each $I_j \subset \mathbb{R}$ is a bounded interval. 
For \(m\le 2\), we prove that \(\Omega\) is uniquely determined by \( \abs{\widehat{\mathbf{1}_\Omega}} \) up to the natural ambiguities of translation and reflection, and we further establish a stability result for this reconstruction.
In contrast, for \(m\ge 3\), uniqueness fails in general.
More precisely, for every \(m\ge 3\), we explicitly construct functions \(f_m,g_m\in\mathcal{I}_m\) such that
$
\abs{\widehat{f_m}}=\abs{\widehat{g_m}},
$
while \(f_m\) cannot be obtained from \(g_m\) by any translation or reflection, where \(\mathcal{I}_m\) denotes the class of indicator functions of unions of exactly \(m\) intervals.

Furthermore, building on the theory of the turnpike problem, in which a finite integer set is uniquely determined by its multiset of pairwise differences under a collision-free condition, we establish an analogous result for finite subsets of $\mathbb{R}$. This, in turn, yields a sufficient condition for recovering indicator functions of finite unions of intervals.
These results provide a complete characterization of the Fourier phase retrieval problem for indicator functions of finite unions of intervals and offer new insights into Fourier phase retrieval for indicator functions of more general domains in higher dimensions.

\end{abstract}
\section{Introduction}

\subsection{Fourier phase retrieval}
{
The Fourier  phase retrieval problem, which seeks to reconstruct a function from the magnitude of its Fourier transform, has attracted considerable attention across diverse fields including crystallography, astronomy, optics, and signal processing \cite{Review,phase1,phase2,phase3}.  While the general phase retrieval problem is known to be highly ill-posed due to inherent ambiguities such as global phase shifts and complex conjugate reflections, the situation may be more favorable when additional structural constraints are imposed on the target functions.

An important special case in Fourier phase retrieval is when the target function is an indicator function
$f=\mathds{1}_{\Omega}$, where $\Omega\subset {\mathbb R}^d$ is a measurable set and $\mathds{1}_{\Omega}$ is the indicator function of $\Omega$, i.e.,
\[
\mathds{1}_{\Omega}(x):=
\begin{cases}
1, & x\in \Omega,\\
0, & x\notin \Omega.
\end{cases}
\]
The aim is to recover $\Omega$, up to a translation  and reflection, from $\abs{\widehat{\mathds{1}_\Omega}}$, where $\widehat{\mathds{1}_\Omega}$ denotes the Fourier transform of $\mathds{1}_\Omega$.
For a convex domain $\Omega \subset \mathbb{R}^d$, the problem of reconstructing $\Omega$ from $\abs{ \widehat{\mathds{1}_\Omega} }$  is equivalent to the covariogram problem  in discrete geometry.
There is an extensive literature on the case where $\Omega$ is a convex domain (see \cite{bianchi1}).

However, there are very few works addressing the case where $\Omega$ is a general set. Among the limited existing studies on general sets, we mention the following contributions. In \cite{bianchi2}, the authors investigate the case where 
$\Omega \in \mathcal{B} \subset \mathbb{R}^2$, where $\mathcal{B}$ 
denotes the class of planar compact sets whose boundary consists of 
a finite number of pairwise disjoint closed simple polygonal curves. 
It is established in \cite[Theorem~1.3]{bianchi2} that if 
$|\widehat{\mathds{1}_\Omega}| = |\widehat{\mathds{1}_K}|$ for some 
convex body $K$ in $\mathbb{R}^2$, then $\Omega$ is necessarily 
convex.  In \cite{bruno}, geometric information extracted from 
$|\widehat{\mathds{1}_\Omega}|$ is studied under the assumption that 
$\Omega$ is a measurable set.
\begin{comment}
 In particular, the covariogram 
$g_\Omega(y) = |\Omega \cap (y + \Omega)|$ is analyzed, and it is 
shown that the directional derivatives of $g_\Omega$ at the origin 
are intimately related to the geometric properties of $\Omega$, 
such as its directional variations and perimeter.
\end{comment}

Many problems in analysis and geometry involve understanding the structural properties of subsets of \(\mathbb{R}^d\). A common approach is to first study the one-dimensional setting, where general sets are often modeled by finite unions of intervals. This type of reduction arises, for example, in the study of Fuglede's conjecture, which asserts that a measurable set \(\Omega \subset \mathbb{R}^d\) is spectral if and only if it tiles \(\mathbb{R}^d\) by translations \cite{Fuglede1974}. In one dimension, it is therefore natural to consider sets \(\Omega\) that are finite unions of intervals. This perspective is reflected in several works, including \L aba's study of Fuglede's conjecture for unions of two intervals \cite{Laba2001}, and the result of Bose and Madan showing that the spectrum of a spectral set consisting of a union of \(n\) intervals must be periodic \cite{BoseMadan2011}. We refer the reader to \cite{FugK} for an alternative proof.

A similar philosophy applies to Fourier phase retrieval for indicator functions. We adopt this one-dimensional perspective and focus on sets \(\Omega \subset \mathbb{R}\) that are finite unions of intervals, investigating the Fourier phase retrieval problem for the corresponding indicator functions \(\mathds{1}_\Omega\). Such functions arise naturally in a variety of applications, including signal processing, where they model compactly supported piecewise constant signals \cite{diffraction}. Despite the fundamental nature of this model, the corresponding Fourier phase retrieval problem has received relatively little attention in the literature. At the same time, finite unions of intervals form the simplest class of non-convex sets and provide a natural testing ground for developing techniques that may eventually extend to more general measurable sets.

\subsection{Problem formulation}

Let $f:\mathbb{R}\rightarrow \{0,1\}$ be a finite union of interval indicator functions, i.e.,
\begin{equation}
\label{eqn: f}
f(x):=\sum_{k=1}^{m}\mathds{1}_{[a_k,b_k]}(x),
\end{equation}
where $m\in \mathbb{N}$ and $\{[a_k,b_k]\}_{k=1}^m$ are pairwise disjoint intervals.
For $m\in\mathbb{N}$, let $\mathcal{I}_m$ denote the set of all such functions with $m$ disjoint intervals:
\[
\mathcal{I}_m:=\left\{f(x)=\sum_{k=1}^{m}\mathds{1}_{[a_k,b_k]}(x) : a_1<b_1<a_2<b_2<\cdots<a_m<b_m\right\}.
\]
We also define
\[
\mathcal{I}:=\bigcup_{m=1}^{\infty}\mathcal{I}_m.
\]
Let $\widehat{f}$ denote the Fourier transform of $f$, i.e.,
$\widehat{f}(\omega)=\int_{-\infty}^{\infty}f(x)\exp(-2\pi \mathrm{i}\cdot \omega\cdot x)dx$.
Through direct computation, we can derive the Fourier transform of $f$ in (\ref{eqn: f}) as:
 \[
 \begin{aligned}
 \widehat{f}(\omega)=&\int_{-\infty}^{\infty}f(x)\exp(-2\pi \mathrm{i}\cdot \omega\cdot x)dx=\sum_{k=1}^{m}\int_{a_k}^{b_k}\exp(-2\pi \mathrm{i} \cdot \omega\cdot x)dx\\
 =&\sum_{k=1}^{m}\frac{\exp(-2\pi \mathrm{i}\cdot b_k\cdot \omega)-\exp(-2\pi \mathrm{i} \cdot a_k\cdot \omega)}{-2\pi \mathrm{i} \cdot \omega},
 \end{aligned}
  \]
for any $\omega\in \mathbb{R}$. Consequently, the absolute value of $\widehat{f}(\omega)$ is expressed as:
  \begin{equation}
  \label{eqn: fourier}
  |\widehat{f}(\omega)|=\abs{\sum_{k=1}^{m}\frac{\exp(-2\pi \mathrm{i}\cdot b_k\cdot \omega)-\exp(-2\pi \mathrm{i}\cdot a_k\cdot \omega)}{2\pi\mathrm{i}\cdot \omega}}.
  \end{equation}
For $f, g\in \cI$, a simple observation is that  $\abs{\widehat{f}} = \abs{\widehat{g}}$ on $\mathbb{R}$ 
if $g$ is obtained from $f$ by either
\begin{enumerate}[(i)]
\item a translation: $g(x) = f(x+c)$ for some $c \in \mathbb{R}$, i.e.,
\[
g(x) = f(x+c)=\sum_{k=1}^{m}\mathds{1}_{[a_k-c,b_k-c]}(x),
\]
or
\item a reflection combined with translation: $g(x) = f(-x+c)$ for some $c \in \mathbb{R}$, i.e.,
\[
g(x) = \sum_{k=1}^{m}\mathds{1}_{[-b_k+c,-a_k+c]}(x).
\]
\end{enumerate}
\begin{comment}
\begin{prop}
Let $f$ be defined as in \eqref{eqn: f}. If $g$ is obtained from $f$ by either
\begin{enumerate}[(i)]
\item a translation: $g(x) = f(x+c)$ for some $c \in \mathbb{R}$, i.e.,
\[
g(x) = \sum_{k=1}^{m}\mathbf{1}_{[a_k+c,b_k+c]}(x),
\]
or
\item a reflection combined with translation: $g(x) = f(-x+c)$ for some $c \in \mathbb{R}$, i.e.,
\[
g(x) = \sum_{k=1}^{m}\mathbf{1}_{[-b_k+c,-a_k+c]}(x),
\]
\end{enumerate}
then $\abs{\widehat{f}} = \abs{\widehat{g}}$ on $\mathbb{R}$.

\end{prop}
\end{comment}
%Therefore, the determination of $f$ from $|\widehat{f}|$ must be defined in a manner that accounts for these two inherent ambiguities.
Therefore, determining $f$ from $\abs{\widehat{f}}$ must account for these two inherent ambiguities.
To make this precise, we introduce the following definition, which formalizes the notion of uniqueness up to these ambiguities.
\begin{defn}\label{def: unique_recovery}
 \begin{enumerate}[(i)] 
 \item We define an {\em equivalence relation} $\sim$ on $\mathcal{I}$ by setting $f \sim g$ if and only if $g$ satisfies either $g(x) = f(x+c)$ or $g(x) = f(-x+c)$ for some constant $c \in \mathbb{R}$ and all $x \in \mathbb{R}$. For any $f\in \mathcal{I}$, we use $[f]$ to denote the equivalence class containing $f$.
  \item An interval indicator function $f\in \mathcal{I}$  is said to be {\em uniquely determined} by $|\widehat{f}|$
  if every interval indicator function $g\in \mathcal{I}$ satisfying $|\widehat{g}| = |\widehat{f}|$ on ${\mathbb R}$
  belongs to  $[f]$, i.e., $g\sim f$. 
  \end{enumerate} 
  \end{defn}

 In this paper, we focus on the following  problems, concerning the uniqueness of recovery and the stability of the reconstruction process:

\begin{enumerate}[{\bf Problem 1.}] \item For a given $f \in \mathcal{I}$, can $f$ be uniquely determined from { $|\widehat{f}|$, the magnitude 
  of its Fourier transform,} alone?

\item Assuming an affirmative answer to Problem 1, is the reconstruction procedure stable? More precisely, does the closeness of $|\widehat{g}|$
  to $|\widehat{f}|$
  imply that $g$ lies in a neighborhood of the equivalence class $[f]$? 
  \end{enumerate}

\subsection{Our contribution}

First of all, we present that $f\in \cI_m$ with $m\leq 2$ can be uniquely determined by $\abs{\widehat{f}}$. 

\begin{thm}
\label{thm: unique_determine}
For any $f\in \cI_m\subset \cI$ with $m\in\{1,2\}$, the function
 $f$ can be uniquely determined by $\abs{\widehat{f}}$  up to equivalence.
 Specifically, if $g\in \cI$ satisfies $\abs{\widehat{f}}=\abs{\widehat{g}}$, then $f\sim g$.
\end{thm}

Moreover, the stability of recovering $f \in \mathcal{I}_m$ with $m \leq 2$ from $|\widehat{f}|$ is established in Theorem~\ref{thm:stable}. For two intervals $I_1, I_2 \subset \mathbb{R}$, define their distance by

$$
\dist(I_1, I_2) := \min_{x \in I_1,\, y \in I_2} |x - y|,
$$
and denote by $\conv(I_1, I_2)$ their convex hull, i.e., the smallest interval containing both $I_1$ and $I_2$. The length of an interval $I \subset \mathbb{R}$ is denoted by $|I|$. For $f, g \in \mathcal{I}$, we define the distance between $f$ and $g$ by
$$
{\dist(f,g) := \min_{\tilde{f} \in [f],\, \tilde{g} \in [g]} \|\tilde{f} - \tilde{g}\|_1},
$$
which is exactly the length of the non-overlapping parts of the corresponding intervals of $f$ and $g$, minimized over all representatives in the equivalence classes $[f]$ and $[g]$.

\begin{thm}\label{thm:stable}
\begin{enumerate}[(i)]
\item Assume that $f=\mathds{1}_{I}\in \cI_1, g= \mathds{1}_{\mytilde{I}}\in \cI_1$ where $I, \mytilde{I}\subset {\mathbb R}$ are intervals.
For any given $\epsilon > 0$,  if $\||\widehat{f}|^2-|\widehat{g}|^2\|_2 \leq  \epsilon,$ then 
\[
\dist(f,g)\leq \sqrt{\frac{3}{4\pi}} \cdot \frac{\epsilon}{\sqrt{\abs{I}}}. 
\]

\item 
Assume that
\[
f = \mathds{1}_{I_1} + \mathds{1}_{I_2} \in \mathcal{I}_2,
\qquad
g = \mathds{1}_{\mytilde{I}_1} + \mathds{1}_{\mytilde{I}_2} \in \mathcal{I}_2,
\]
where \(I_1, I_2, \mytilde{I}_1, \mytilde{I}_2 \subset \mathbb{R}\) are intervals such that
$
|I_1| \le |I_2|
\,\text{and}\,
|\mytilde{I}_1| \le |\mytilde{I}_2|.
$
Suppose that
\[
\bigl\||\widehat{f}|^2 - |\widehat{g}|^2\bigr\|_2 \le \epsilon
\]
for some \(\epsilon > 0\) satisfying
\begin{equation}
\label{eqn: epsilon}
0 < \epsilon^2 < \min\left\{\frac{\pi\cdot c^3}{12}, \frac{\pi\cdot c_0^2\cdot c}{576}, \frac{\pi\cdot c_0^3}{3456}\right\},
\end{equation}
where \(c\) and \(c_0\) are positive constants such that
\[
c \le \min\{|I_1|, |\mytilde{I}_1|, |I_2|, |\mytilde{I}_2|\}
\quad\text{and}\quad
c_0 \le \min\{\dist(I_1, I_2), \dist(\mytilde{I}_1, \mytilde{I}_2)\}.
\]
Then we have
\begin{equation}\label{eq:I1I2}
\begin{aligned}
&\bigl||I_1| - |\mytilde{I}_1|\bigr|
  \lesssim \frac{\epsilon}{\sqrt{\min\{c_0,c\}}},\qquad
\bigl||I_2| - |\mytilde{I}_2|\bigr|
  \lesssim \frac{\epsilon}{\sqrt{\min\{c_0,c\}}},\\
&\bigl|\dist(I_1, I_2) - \dist(\mytilde{I}_1, \mytilde{I}_2)\bigr|
  \lesssim \frac{\epsilon}{\sqrt{\min\{c_0,c\}}}.
  \end{aligned}
\end{equation}
In particular, this implies
\begin{equation}
\label{eqn: distance_result1}
\dist(f,g) \lesssim \frac{\epsilon}{\sqrt{\min\{c_0,c\}}}.
\end{equation}

\end{enumerate}
\end{thm}

\begin{rem}
{ Although the stability of functions depending on the modulus of the Fourier transform is studied  in \cite{stable_fourier}, no precise characterization is provided for specific function classes. In this regard, a stability estimate for the endpoints of the support interval of an indicator function  $f\in \cI_2$ is established in the present work.  Here (\ref{eq:I1I2}) implies that the endpoints of $I_1$ are 
close to those of $\widetilde{I}_1$, and simultaneously the endpoints 
of $I_2$ are close to those of $\widetilde{I}_2$. A straightforward 
observation shows that (\ref{eq:I1I2}) implies 
(\ref{eqn: distance_result1}). However, (\ref{eq:I1I2}) is 
strictly stronger than \eqref{eqn: distance_result1}, as the converse 
does not hold in general. We illustrate this distinction with the 
following concrete example.\\
\indent Consider $I_1 = [0, 1-c]$, $I_2 = [1, 2]$, 
$\mytilde{I}_1 = [0, c]$, and $\mytilde{I}_2 = [2c, 2]$ for some 
$c \in (0, 0.1)$, and let $f = \mathds{1}_{I_1} + \mathds{1}_{I_2}$ 
and $g = \mathds{1}_{\mytilde{I}_1} + \mathds{1}_{\mytilde{I}_2}$. 
Setting $\epsilon = c^{3/2}$, we have $\dist(f, g) = 2c = 2\epsilon / 
\sqrt{c}$. Nevertheless,
\[
\bigl||I_1| - |\mytilde{I}_1|\bigr| = 1 - 2c \gg 2\epsilon / 
\sqrt{c},
\]
when $c$ is sufficiently small enough. 
%which is not of the same order as $\epsilon / \sqrt{c}=c$ as 
%$c \to 0$.
 }
\end{rem}
\begin{comment}
\begin{rem}
The stability result is restricted to small $\epsilon$ because when $\epsilon$ is large, the parameters of the indicator function cannot be estimated stably.
  For example, take 
\[
f(x)=\mathds{1}_{[a_1,b_1]}(x)+\mathds{1}_{[a_2,b_2]}(x),\quad\text{and}\quad g_c(x)=\mathds{1}_{[\widetilde{a}_1,\mytilde{b}_1]}(x)+\mathds{1}_{[\mytilde{a}_2+c,\mytilde{b}_2+c]}(x),
\]
such that $a_1<b_1<a_2<b_2$, $\mytilde{a}_1<\mytilde{b}_1<\mytilde{a}_2<\mytilde{b}_2$, and
\[
\max\{b_2-a_2,b_1-a_1,\mytilde{b}_2-\mytilde{a}_2,\mytilde{b}_1-\mytilde{a}_1\}<a_2-b_1<b_2-a_1<\mytilde{a}_2-\mytilde{b}_1.
\]
Based on the Fourier transform property in (\ref{eqn: f_property}) and the formulation in (\ref{eqn: f_formulation_origin}), we have
\[
\||\widehat{f}|^2-|\widehat{g}_0|^2\|_2=\||\widehat{f}|^2-|\widehat{g}_c|^2\|_2
\]
for any $c\geq 0$. Let $\epsilon:=\||\widehat{f}|^2-|\widehat{g}_0|^2\|_2$. Then the interval spacing parameter $c$ cannot be estimated stably based on $\epsilon$ alone.
\end{rem}
\end{comment}
\begin{rem}
In part (ii) of Theorem \ref{thm:stable}, the stability result is restricted to the small-$\epsilon$ regime specified in (\ref{eqn: epsilon}), since the location parameters of the indicator function cannot be estimated stably when $\epsilon$ is large.
For example, let
\[
f(x)=\mathds{1}_{[0,1]}(x)+\mathds{1}_{[3,4]}(x),
\qquad
g_c(x)=\mathds{1}_{[0,1]}(x)+\mathds{1}_{[6+c,7+c]}(x),
\]
for some $c\geq 0.$ Using the Fourier transform property in \((\ref{eqn: f_property})\), we obtain
\[
\left\|\abs{\widehat f}^2-\abs{\widehat {g_c}}^2\right\|_2
=
\left\|\abs{\widehat f}^2-\abs{\widehat {g_0}}^2\right\|_2
=
\frac{2\sqrt{6}}{3}=:\epsilon
\]
for every \(c\geq 0\). Nevertheless, as \(c\to\infty\), the second interval of
\(g_c\) moves arbitrarily far away from that of \(f\). Hence the Fourier
magnitude data remain at a fixed distance from those of \(f\), while the
underlying location parameter becomes unbounded. This shows that such a
stability statement in (\ref{eq:I1I2}) cannot hold for large \(\epsilon\).
\end{rem}

\begin{rem}
Although the unique recovery result for the two-interval indicator function case can be obtained by setting $\epsilon=0$ in Theorem
\ref{thm:stable}, the result differs slightly from the presentation in Theorem \ref{thm: unique_determine}. Specifically, Theorem \ref{thm: unique_determine} does not assume the number of intervals for $g$, whereas Theorem \ref{thm:stable} treats the number of intervals for $g$ as prior information.
\end{rem}

Besides, the following theorem shows that, for $f \in \mathcal{I}_m$
  with $m \ge 3$, uniqueness of recovery from $\abs{\widehat{f}}$
  cannot, in general, be guaranteed.

\begin{thm}\label{th:fanli}
For any integer $m\geq 3$, there exist functions  $f_m,g_m\in \cI_m$ such that $\abs{\widehat{f_m}(\omega)}=\abs{\widehat{g_m}(\omega)}$ for all $\omega\in {\mathbb R}$ but $f_m\not\sim g_m$.
\end{thm}
{Constructing explicit examples for arbitrary $m$ is non-trivial; 
the analysis relies crucially on certain special properties arising 
from $f_m \not\sim g_m$, and we refer the reader to 
Section~\ref{sec:disF} for details. That section may also be read 
independently. Furthermore, imposing additional structural 
conditions on the functions in $\mathcal{I}_m$, unique recovery can 
still be guaranteed. Motivated by the \emph{collision-free 
condition} from the \emph{turnpike problem}, we introduce in 
Section~\ref{sec:sep} a sufficient condition on the intervals 
$I_1, \ldots, I_m$, termed the \emph{separation condition} 
(Definition~\ref{de:sep}), which ensures the unique recovery of 
$f = \sum_{j=1}^{m} \mathds{1}_{I_j}$ from $|\widehat{f}|$; 
see Theorem~\ref{thm: separation_condition} for the precise statement.
}
  
  \begin{comment}
Though $f \in \mathcal{I}$ cannot be uniquely determined by $|\widehat{f}|$, we can still extract some geometric information about $f$.  
  \textcolor{red}{We may add a new theorem which state which geometry information we can obtain by $\abs{\widehat{f}}$.}
  \textcolor{red}{
  \begin{thm}
  Assume that $f\in \cI$ in the form of $f=\sum_{j=1}^m \mathds{1}_{I_j}$.  We can obtain the following information by  $\abs{\widehat{f}}$:
  \begin{enumerate}
  \item the number of intervals, i.e., $m$.
  \item the length of $\conv(I_1,\ldots,I_m)$
  \item $\max\{\abs{I_1},\ldots,\abs{I_m}\}$
  \item $\max\{\dist(I_j, I_{j+1}), j=1,\ldots,m-1\}$...
  \end{enumerate}
  \end{thm}
   }
\end{comment}
\subsection{Related work}

\subsubsection{The covariogram problem }
{As mentioned earlier, for a convex domain $\Omega \subset \mathbb{R}^d$, 
the problem of reconstructing $\Omega$ from $|\widehat{\mathds{1}_\Omega}|$ 
is equivalent to the covariogram problem in convex geometry. Indeed, 
by Parseval's identity, the covariogram $g_\Omega(y) = {\rm vol}(\Omega \cap (\Omega + y))$ 
with $y\in \mathbb{R}^d$
satisfies $\widehat{g_\Omega} = |\widehat{\mathds{1}_\Omega}|^2$, so 
$g_\Omega$ is uniquely determined by $|\widehat{\mathds{1}_\Omega}|$, 
and vice versa. Consequently, recovering $\Omega$ from 
$|\widehat{\mathds{1}_\Omega}|$ is equivalent to recovering $\Omega$ 
from its covariogram $g_\Omega$.
There is an extensive literature devoted to the case where $\Omega$ 
is a convex body (see \cite{bianchi1}). We briefly summarize the 
current state of the art, where ``determined'' always means that 
$\Omega$ is uniquely determined by $|\widehat{\mathds{1}_\Omega}|$ 
among all convex bodies, up to translation and reflection through 
the origin.\\
\indent When $\Omega \subset \mathbb{R}^2$ is convex, Averkov and Bianchi 
proved in \cite{omega1} that $\Omega$ is always determined, thereby 
confirming \textit{Matheron's conjecture} that every 
convex body in $\mathbb{R}^2$ is determined by its covariogram among 
all convex bodies, up to translation and reflection through the origin.
} For $\Omega \subset \mathbb{R}^3$, Bianchi proved in \cite{omega2} that every three-dimensional convex polyhedron is determined. In dimensions $d \geq 4$, Bianchi also showed by explicit construction that there exist convex polytopes $\Omega \subset \mathbb{R}^d$
  that are not determined \cite{omega3}.  However, it remains an open problem whether all convex bodies $\Omega \subset \mathbb{R}^3$
  are determined.

\subsubsection{The classical Fourier phase retrieval}

The classical Fourier phase retrieval problem aims to recover a compactly supported function $f\in L^2(\mathbb{R}^d)$
from the magnitude $\abs{\widehat{f}}$. In \cite[Theorem 4.6]{Review}, a general relationship between functions $f$ and $g$ is established under the assumption that they are compactly supported elements of $L^2(\mathbb{R}^d)$
satisfying $\abs{\widehat{f}} = \abs{\widehat{g}}$.

\begin{thm}[Theorem 4.6 in \cite{Review}]\label{th:fourierpr}
Let $f,g\in L^2(\mathbb{R}^{d})$ be compactly supported and let $F,G$ denote the Fourier-Laplace transform of $f$ and $g$ respectively. Then $\abs{\widehat{f}}=\abs{\widehat{g}}$ if and only if  there exists a factorization $G=G_1\cdot G_2$, a constant $\gamma$ with $|\gamma|=1$, and an entire function $Q$ where $Q|_{\mathbb{R}^{d}}$ is real-valued such that 
\[
F(z)=\gamma \cdot e^{\mathrm{i}Q(z)}\cdot G_1(z)\cdot \overline{G_2(\overline{z})}. 
\]
\end{thm}

In the one-dimensional setting, a substantially more explicit description of all possible ambiguities is available. The uniqueness question for functions of one variable was already investigated in the late 1950s by Akutowicz \cite{FR1, FR2} and, independently a few years later, by Walther \cite{phase2} and Hofstetter \cite{phase1}. Their results show that in one dimension every nontrivial ambiguity arises from a specific transformation of the zeros of the holomorphic extension of the Fourier transform. Thus, for a single variable, the structure of all ambiguous solutions is completely characterized in terms of such ``zero-flipping" operations. See  \cite{Review} for more details.
A discrete analogue of Theorem \ref{th:fourierpr} was also established in \cite{Review}; we recall this result in Section \ref{sec:disF}. Furthermore, in the real-valued setting, we obtain a more explicit characterization of the ambiguities (see Theorem \ref{th:realfac}), which plays a key role in the proof of Theorem \ref{th:fanli}.

However, when $f$ and $g$ are constrained to be interval indicator functions, the analysis must be refined based on the particular structure inherent to such functions.

\subsubsection{The Separation Condition}

\begin{comment}
In \cite{fienup}, the authors are given $m$ predetermined intervals $I_j$ , $j=1,\ldots,m$, that satisfy the {\em separation condition} (see Definition \ref{de:sep} ), and investigate whether a function $f$ can be uniquely determined from the magnitude of its Fourier transform, under the assumption that the support of $f$ lies within these known intervals $I_j, j=1,\ldots,m$. The authors of  \cite{fienup} show that $f$ is almost always essentially the only solution with support contained in $\cup_{j=1}^mI_j$.

In this paper, we still employ  the {\em separation condition} (Definition \ref{de:sep}), that guarantees the unique recovery of $f=\sum_{j=1}^m \mathds{1}_{I_j}$
  from $\abs{\widehat{f}}$  (Theorem \ref{thm: separation_condition}).
 However, there is a fundamental difference in problem setup between \cite{fienup} and our work. In \cite{fienup}, the authors are given $m$ predetermined intervals $I_j$, $j=1,\ldots,m$, that satisfy the separation condition, and investigate whether a function $f$ can be uniquely determined from the magnitude of its Fourier transform, under the assumption that the support of $f$ lies within these known intervals $I_j$, $j=1,\ldots,m$. In contrast, our work focuses on recovering indicator functions of unknown intervals that satisfy the separation condition, but without any prior knowledge of what these intervals are.
% \end{rem}

\end{comment}
 In \cite{fienup}, the authors consider \(m\) predetermined intervals \(I_j\), \(j = 1,\ldots,m\), that satisfy the separation condition (see Definition~\ref{de:sep}), and study whether a function \(f\) can be uniquely determined from the magnitude of its Fourier transform, under the assumption that the support of \(f\) is contained in the known union \(\bigcup_{j=1}^m I_j\). They show that, under this assumption, \(f\) is almost always essentially the only function with support contained in \(\bigcup_{j=1}^m I_j\) that matches the given magnitude data.

 In this paper, we also employ the separation condition (Definition~\ref{de:sep}), which guarantees the unique recovery of
 \[
 f = \sum_{j=1}^m \mathds{1}_{I_j}
 \]
 from \(\lvert \widehat{f} \rvert\) (Theorem~\ref{thm: separation_condition}). However, there is a fundamental difference in problem setup between \cite{fienup} and the present work. In \cite{fienup}, the intervals \(I_j\) are given in advance and are assumed to satisfy the separation condition, and the question is whether a function \(f\) supported in \(\bigcup_{j=1}^m I_j\) is uniquely determined by \(\lvert \widehat{f} \rvert\). In contrast, in our setting the intervals themselves are not known a priori: we aim to recover indicator functions of intervals that satisfy the separation condition, but without any prior information about the locations of these intervals.

\subsection{Organization of the paper}

The paper is organized as follows. In Section~\ref{sec: proof of Thm 1.2}, we prove Theorem~\ref{thm: unique_determine}, and in Section~\ref{sec:3}, we prove Theorem~\ref{thm:stable}. Section~\ref{sec:disF} is divided into two parts. The first establishes a discrete Fourier phase retrieval result for real-valued signals, which is subsequently used to construct counterexamples to Theorem~\ref{th:fanli}. The second introduces the turnpike problem, whose central question asks whether a finite set is uniquely determined by its multiset of pairwise differences; building on the classical result that a finite integer set is uniquely determined under a collision-free condition, we establish an analogous result, Theorem~\ref{thm: real_result}, for finite subsets of $\mathbb{R}$, which will be employed in Section~\ref{sec: sec6}. In Section~\ref{sec: sec5}, we apply the results of Section~\ref{sec:disF} to prove Theorem~\ref{th:fanli}. Finally, Section~\ref{sec: sec6} introduces a separation condition and establishes that, under this condition, every $f \in \mathcal{I}_m$ is uniquely determined by $\lvert\widehat{f}(\omega)\rvert$ for all $\omega \in \mathbb{R}$.

\section{Proof of Theorem \ref{thm: unique_determine}}
\label{sec: proof of Thm 1.2}

In this section, we present the proof of Theorem~\ref{thm: unique_determine}. We remark that the proof could be simplified by using the turnpike problem result, which guarantees that four real points are uniquely determined by their multiset of pairwise differences. Nevertheless, we retain the current self-contained proof for the sake of completeness. It should be emphasized that the present problem is not identical to the classical turnpike problem: differences between two left endpoints (or two right endpoints) and differences between a left endpoint and a right endpoint are naturally associated with different signs and are therefore distinguishable.

\begin{proof}[Proof of Theorem \ref{thm: unique_determine}]
 Assume that $g \in \mathcal{I}_{m'}$ satisfies $\abs{\widehat{f}} = \abs{\widehat{g}}$
 for some integer $m' \geq 1$.
 Without loss of generality, we assume that 
 \[
 f(x)=\sum_{k=1}^{m}\mathds{1}_{[a_k,b_k]}(x),\quad \text{and}\quad g(x)=\sum_{k=1}^{m'}\mathds{1}_{[a'_k,b'_k]}(x).
 \]
We first show that  $m=m'$. From  $\abs{\widehat{f}}=\abs{\widehat{g}}$, we can directly derive, for any $\omega\in {\mathbb R}\setminus \{0\}$:
\[
\begin{aligned}
  \abs{\widehat{f}(\omega)}=&\abs{\sum_{k=1}^{m}\frac{\exp(-2\pi \mathrm{i}\cdot b_k\cdot \omega)-\exp(-2\pi \mathrm{i}\cdot a_k\cdot \omega)}{2\pi \omega}}\\
  =&\abs{\sum_{k=1}^{{m'}}\frac{\exp(-2\pi \mathrm{i}\cdot {b_k'}\cdot \omega)-\exp(-2\pi \mathrm{i}\cdot {a_k'}\cdot \omega)}{2\pi \omega}}= \abs{\widehat{g}(\omega)}.
  \end{aligned}
\]
Consequently, for any $t\in \mathbb{R}$, we have
\begin{equation}
\label{eqn: temp}
\left|\sum_{k=1}^{m}\left({\exp(\mathrm{i}\cdot b_k\cdot t)-\exp(\mathrm{i}\cdot a_k\cdot t)}\right)\right|^2=\left|\sum_{k=1}^{{m'}}\left({\exp( \mathrm{i}\cdot {b_k'}\cdot t)-\exp(\mathrm{i}\cdot {a_k'}\cdot t)}\right)\right|^2.
\end{equation}
Setting 
\begin{equation}
\label{eqn: c}
\{c_k\}_{k=1}^{2m} = \{a_1, b_1, \ldots, a_m, b_m\}
\end{equation}
 with $c_1 > c_2 > \cdots > c_{2m}$, and similarly $\{c'_k\}_{k=1}^{2m'} = \{a'_1, b'_1, \ldots, a'_{m'}, b'_{m'}\}$ with $c'_1 > c'_2 > \cdots > c'_{2m'}$, we obtain from \eqref{eqn: temp}:
\[
\begin{aligned}
\left|\sum_{k=1}^{m}\left(\exp{(\mathrm{i} \cdot b_k\cdot  t)} - \exp{(\mathrm{i} \cdot a_k \cdot t)}\right)\right|^2 
&= 2m + 2\sum_{i < j}(-1)^{i+j} \cos\bigl((c_i - c_j)\cdot t\bigr), \\
\left|\sum_{k=1}^{m'}\left(\exp{( \mathrm{i} \cdot b'_k \cdot t)} - \exp{( \mathrm{i} \cdot  a'_k \cdot t)}\right)\right|^2 
&= 2m' + 2\sum_{i < j}(-1)^{i+j} \cos\bigl( (c'_i - c'_j)\cdot t\bigr),
\end{aligned}
\]
which implies
\begin{equation}
\label{eqn: temp_final}
(m-{m'})+\sum_{i< j}(-1)^{(i+j)}\cdot \cos((c_i-c_j)\cdot t)-\sum_{i< j}(-1)^{(i+j)}\cdot \cos((c'_i-c'_j)\cdot t)=0,
\end{equation}
for any $t\in \mathbb{R}$. 
Recall that the function set 
\[
\{1,\cos(d_1\cdot t),\ldots,\cos(d_n\cdot t): 0<d_1<d_2<\cdots<d_n\}
\]
 for any given positive interger $n$ forms a Chebyshev system \cite{chebyshev}. Therefore, we can directly conclude that $m=m'$.

Now we begin to prove the results case by case. 

\textbf{Case 1: $m=1$. }

By Definition \ref{def: unique_recovery}, both intervals may be assumed to start at the origin without loss of generality. It therefore suffices to show the following: if $f$ and $g$ are interval indicator functions of the form
\[
f(x) = \mathds{1}_{[0,b_1]}(x) \quad \text{and} \quad g(x) = \mathds{1}_{[0,{b'_1}]}(x),
\]
where $b_1,b'_1>0$, then $\abs{\widehat{f}} = \abs{\widehat{g}}$ implies $b_1 = b'_1$. Since the formulation in \eqref{eqn: temp_final} reduces to:
\[
\cos(b_1\cdot t)-\cos(b'_1\cdot t)=0,
\]
for any $t\in \mathbb{R}$, we immediately obtain $b_1 = b_1'$, and the conclusion follows.

\textbf{Case 2: $m=2$. }
Again by Definition~\ref{def: unique_recovery}, let $f$ and $g$ be given by
\[
f(x) = \mathds{1}_{[0,b_1]}(x) + \mathds{1}_{[a_2,b_2]}(x) \quad \text{and} \quad g(x) = \mathds{1}_{[0,b'_1]}(x) + \mathds{1}_{[a'_2,b'_2]}(x).
\]
Without loss of generality, assume that:
\begin{equation}\label{eq:order}
\begin{aligned} & 0 < b_1 < a_2 < b_2 \quad \text{with} \quad b_2 - a_2 \geq b_1, 
 \text{ and } 0 < b'_1 < a'_2 < b'_2 \quad \text{with} \quad b'_2 - a'_2 \geq b'_1.
  \end{aligned}
\end{equation}
With the notation established in (\ref{eqn: c}), we have
\[
c_1=b_2, \quad c_2=a_2,\quad c_3=b_1,\quad c_4=0,
\]
and 
 \[
 c'_1=b'_2,\quad c'_2=a'_2, \quad c'_3=b'_1, \quad c'_4=0.
 \] 
  It suffices to show that if  $|\widehat{f}| = |\widehat{g}|$, then  $c_j=c'_j,$  $j=1,2,3,4$.
%The subsequent proof is divided into two steps.
%\textbf{The differences $c_i - c_j$ are pairwise distinct for $1 \leq i < j \leq 4$.}
\begin{comment}
Based on (\ref{eqn: temp_final}) and whether $i+j$ is even or odd,
we have
\[
\{c_1-c_3,c_2\}=\{c'_1-c'_3, c'_2\},\text{ and } \{c_1-c_2,c_1,c_2-c_3,c_3\}=\{c'_1-c'_2,c'_1,c'_2-c'_3,c'_3\}.
\]
Here, we use $c_4=c_4'=0$.
According to (\ref{eq:order}), we have
\[
c_1=\max \{c_1-c_2,c_1,c_2-c_3,c_3\}=\max\{c'_1-c'_2,c'_1,c'_2-c'_3,c'_3\}=c'_1.
\]

Since $\{c_1-c_3,c_2\}=\{c'_1-c'_3, c'_2\}$, either $c_1-c_3=c'_1-c'_3$
and $c_2=c_2'$, or $c_1-c_3=c_2'$  and $c_2=c'_1-c'_3$. The first alternative immediately gives $c_1=c'_1$
 and $c_2=c_2'$, establishing the result. We proceed to rule out the second alternative.
 The equalities 
 \[
 c_1-c_3=c_2',\qquad c_2=c'_1-c'_3
  \]
  imply 
  \[
  c_1-c_2-c_3=c_2'-c_1'+c_3'
  \]
  . However, according to (\ref{eq:order}), we have $c_1-c_2-c_3\geq 0$ while $c_2'-c_1'+c_3'\leq 0$, which implies $c_2+c_3=c_1$ and $c_2'+c_3'=c_1'=c_1$.
\end{comment}
The subsequent proof is divided into three steps.

\textbf{Substep 1: Prove that $b_2=b'_2$. }

Consider the sets $\{c_i - c_j\}_{1 \leq i < j \leq 4}$ and $\{c_i' - c_j'\}_{1 \leq i < j \leq 4}$ appearing in \eqref{eqn: temp_final}. Note that the only pair $(c_i, c_j)$ with $i < j$ satisfying $c_i - c_j = b_2$ is $(c_1, c_4)$, and similarly, the only pair $(c_i', c_j')$ with $i < j$ satisfying $c_i' - c_j' = b_2'$ is $(c_1', c_4')$.

Suppose for contradiction that $b_2 \neq b_2'$; without loss of generality, assume $b_2 > b_2'$. Applying the properties of Chebyshev systems to \eqref{eqn: temp_final}, we deduce that $\cos(b_2 \cdot t) \equiv 0$, which is a contradiction since $\cos(b_2 \cdot t)$ cannot vanish identically for any $b_2 \in \mathbb{R}$. Therefore $b_2 = b_2'$.

\textbf{Substep 2: Prove that $b_1=b'_1$. }

We begin by establishing the following identity 
\begin{equation}
\label{eqn: sum_temp}
b_2 - a_2 + b_1 = b_2' - a_2' + b_1'.
\end{equation}
By the fact that, for any $t\in \mathbb{R}$, 
\[
\cos(t)=\sum_{n=0}^{\infty}\frac{(-1)^n}{(2n)!} \cdot t^{2n},%\qquad \text{and}\qquad \sin(t)=\sum_{n=0}^{\infty}\frac{(-1)^nt^{2n+1}}{(2n+1)!},
\]
thus
\[
\begin{aligned}
&\sum_{i< j}(-1)^{(i+j)}\cdot \cos(({c}_i-{c}_j)\cdot t)
%=&-2\cos((b_2-a_2)\cdot x)+2\cos((b_2-b_1)\cdot x)-2\cos((a_2-b_1)\cdot x)-2\cos(b_2\cdot x)+2\cos(a_2\cdot x)-2\cos(b_1\cdot x)\\
 = \sum_{n=0}^{\infty}\alpha_{n}t^{2n},\ \text{and}\  \sum_{i< j}(-1)^{(i+j)}\cdot \cos((c'_i-c'_j)\cdot t)= \sum_{n=0}^{\infty}\alpha'_{n}t^{2n},
\end{aligned}
\]
%\[
%\begin{aligned}
%&\sum_{i< j}(-1)^{(i+j)}\cdot \cos((\mytilde{c}_i-\mytilde{c}_j)\cdot t)\\
%=&-2\cos((\mytilde{b}_2-\mytilde{a}_2)\cdot x)+2\cos((\mytilde{b}_2-\mytilde{b}_1)\cdot x)-2\cos((\mytilde{a}_2-\mytilde{b}_1)\cdot x)-2\cos(\mytilde{b}_2\cdot x)+2\cos(\mytilde{a}_2\cdot x)-2\cos(\mytilde{b}_1\cdot x)\\
 %=& \sum_{n=0}^{\infty}\mytilde{\alpha}_{n}x^{2n},
%\end{aligned}
%\]
where 
\begin{equation}
\label{eqn: alpha}
\alpha_{n}:=\frac{(-1)^n\left(-2(b_2-a_2)^{2n}+2(b_2-b_1)^{2n}-2(a_2-b_1)^{2n}-2b_2^{2n}+2a_2^{2n}-2b_1^{2n}\right)}{(2n)!},
\end{equation}
and 
\begin{equation}
\label{eqn: w_alpha}
%\mytilde{\alpha}_{n}:=\frac{(-1)^n\left(-2(\mytilde{b}_2-\mytilde{a}_2)^{2n}+2(\mytilde{b}_2-\mytilde{b}_1)^{2n}-2(\mytilde{a}_2-\mytilde{b}_1)^{2n}-2\mytilde{b}_2^{2n}+2\mytilde{a}_2^{2n}-2\mytilde{b}_1^{2n}\right)}{(2n)!}.
\alpha_{n}':=\frac{(-1)^n\left(-2(b_2'-a_2')^{2n}+2(b_2'-b_1')^{2n}-2(a_2'-b_1')^{2n}-2(b'_2)^{2n}+2(a'_2)^{2n}-2(b'_1)^{2n}\right)}{(2n)!}.
\end{equation}
Substituting these into (\ref{eqn: temp_final}) with $m={m'}$, it obtains that 
\begin{equation}
\label{eqn: alpha_result}
\alpha_{n}=\alpha'_{n}, \quad \text{for}\ n=0,1,2,\cdots.
\end{equation}
Take $n=1$, we have
\[
\alpha_{1}={(b_2-a_2)^2-(b_2-b_1)^{2}+(a_2-b_1)^{2}+b_2^2-a_2^2+b_1^{2}}=(b_2-a_2+b_1)^2.
\]
Similarly,
\[
\alpha'_{1}={(b'_2-a'_2)^2-(b'_2-b'_1)^{2}+(a'_2-b'_1)^{2}+b_2^{\prime2}-a_2^{\prime2}+b_1^{\prime2}}=(b'_2-a'_2+b'_1)^2.
\]
Since $b_2-a_2+b_1>0$ and $b'_2-a'_2+b'_1>0$, combined with $\alpha_1=\alpha'_1$, this directly yields the conclusions in (\ref{eqn: sum_temp}).

We next  show  $b_1=b'_1$. For convenience, we set
\begin{equation}\label{eq:sb}
s := b_1, \quad s' := b'_1, \quad \beta := b_2 = b'_2, \quad \gamma := b_2 - a_2 + b_1 = b'_2 - a'_2 + b'_1.
\end{equation}
Here, we use (\ref{eqn: sum_temp}) and the result in Substep 1. 
 Consequently, we can express $a_2$ and $a'_2$ as 
 \begin{equation}\label{eq:a2a2}
 a_2 = \beta - \gamma + s\quad \text{and}\quad a'_2 = \beta - \gamma + s',
 \end{equation}
  respectively. 
Furthermore, according to (\ref{eqn: alpha}), (\ref{eqn: w_alpha}) and (\ref{eq:sb}), $\alpha_2$ and $\alpha'_2$ can be reformulated as:
\[
\begin{aligned}
\alpha_{2}=\frac{-(s-\gamma)^4+(s-\beta)^4-(\beta-\gamma)^4-\beta^4+(s-\gamma+\beta)^4-s^4}{12},
\end{aligned}
\]
and 
\[
\begin{aligned}
\alpha'_{2}=\frac{-(s'-\gamma)^4+(s'-\beta)^4-(\beta-\gamma)^4-\beta^4+(s'-\gamma+\beta)^4-(s')^4}{12}.
\end{aligned}
\]
Given that $\alpha_2 = \alpha'_2$, it leads to:
\[
-(s-\gamma)^4+(s-\beta)^4+(s-\gamma+\beta)^4-s^4=-(s'-\gamma)^4+(s'-\beta)^4+(s'-\gamma+\beta)^4-(s')^4.
\]
Following from $b_2 - a_2 \geq b_1$ and $b'_2 - a'_2 \geq b_1$, it leads to $s\leq \gamma/2$ and $s'\leq \gamma/2$. To establish that, $b_1 = b'_1$, or equivalently, $s = s'$, it suffices to prove that the function $h:\mathbb{R}\rightarrow \mathbb{R}$ defined as
\[
h(s) := -(s-\gamma)^4+(s-\beta)^4+(s-\gamma+\beta)^4-s^4
\] 
is monotonic on the interval $s \in (0,\gamma/2]$.
Through direct computation, we obtain the derivative of $h$:
 \[
 \small
 \begin{aligned}
\frac{\mathrm{d}h(s)}{\mathrm{d}s}
\,\,=\,\,12(\beta-\gamma)\cdot\beta\cdot (2s-\gamma)\,\,\leq\,\, 0,
 \end{aligned}
 \]
 for any $s\in (0,\gamma/2]$, based on the fact that $\beta\geq \gamma>0$. Therefore, $h$ is monotonically decreasing on this interval, which implies $s = s'$, or equivalently, $b_1 = b'_1$.

 \begin{comment}
 \[
 \small
 \begin{aligned}
 h'(s)=&-4(s-\gamma)^3+4(s-\beta)^3+4(s-\gamma+\beta)^3-4s^3\\
 =&4((\gamma-\beta)\cdot ((s-\beta)^2+(s-\gamma)^2+(s-\beta)\cdot (s-\gamma))\\
 &+(\beta-\gamma)\cdot ( (s-\gamma+\beta)^2+s^2+s\cdot (s-\gamma+\beta)))\\
 =&4(\beta-\gamma)\cdot (-s^2-\beta^2+2s\cdot\beta-s^2-\gamma^2+2s\cdot \gamma-s^2+(\beta+\gamma)\cdot s-\beta\cdot \gamma)\\
  &+(\beta-\gamma)\cdot (s^2+\gamma^2+\beta^2-2s\cdot \gamma +2s\cdot \beta-2\gamma\cdot \beta+s^2+s^2-s\cdot \gamma+s\cdot \beta)\\
   =&12(\beta-\gamma)\cdot\beta\cdot (2s-\gamma)\leq 0,
 \end{aligned}
 \]
\end{comment}
 
 \textbf{Substep 3: Prove that $a_2=a'_2$. }
 
Based on the equation $b_2 - a_2 + b_1 = b'_2 - a'_2 + b'_1$ in (\ref{eqn: sum_temp}), and given that $b_1 = b'_1$ and $b_2 = b'_2$, we can conclude that $a_2 = a'_2$. This completes the proof.
\end{proof}
%\subsection{Stability Results}
\section{Proof of Theorem \ref{thm:stable}}
\label{sec:3}
The following formula will be used repeatedly in the sequel. Although it follows from a straightforward computation, we provide the proof here for completeness.

\begin{lem}
\label{lem: key}
Let $a,b,c,d\in \mathbb{R}$ with $a<b$ and $c<d$. Set
\[
L_1:=b-a,\qquad L_2:=d-c.
\]
Then
\begin{equation}
\label{eqn: ab_cd}
\left(\mathds{1}_{[a,b]}\mycircledast \mathds{1}_{[c,d]}\right)(x)
=
\begin{cases}
x-(a+c), 
& \text{if } x\in [a+c,\,a+c+\min\{L_1,L_2\}],\\
\min\{L_1,L_2\}, 
& \text{if } x\in [a+c+\min\{L_1,L_2\},\,a+c+\max\{L_1,L_2\}],\\
a+c+L_1+L_2-x, 
& \text{if } x\in [a+c+\max\{L_1,L_2\},\,b+d],\\
0, 
& \text{otherwise}.
\end{cases}
\end{equation}
\end{lem}
\begin{proof}
By the definition of convolution, we have
\[
\left(\mathds{1}_{[a,b]}\mycircledast \mathds{1}_{[c,d]}\right)(x)
=
\int_{\mathbb{R}} \mathds{1}_{[a,b]}(y)\mathds{1}_{[c,d]}(x-y)\,dy.
\]
The integrand is equal to $1$ precisely when
\[
y\in [a,b]
\quad\text{and}\quad
x-y\in [c,d],
\]
or equivalently,
\[
y\in [a,b]\cap [x-d,x-c].
\]
Therefore,
\[
\left(\mathds{1}_{[a,b]}\mycircledast \mathds{1}_{[c,d]}\right)(x)
=
\left|[a,b]\cap [x-d,x-c]\right|.
\]
It remains to compute the length of this intersection.

After translating the two intervals by $-a$ and $-c$, this length is the same as
\[
\big|[0,L_1]\cap [x-(a+d),\,x-(a+c)]\big|=\big|[0,L_1]\cap [x-(a+c)-L_2,\,x-(a+c)]\big|.
\]
%Equivalently, it is the length of the overlap of two intervals of lengths
%$L_1$ and $L_2$ whose relative displacement is $x-(a+c)$. (the relative displacement is a bit wierd).
Thus the overlap length increases linearly from $0$ to
$\min\{L_1,L_2\}$ as $x$ goes from $a+c$ to
$a+c+\min\{L_1,L_2\}$, then remains equal to
$\min\{L_1,L_2\}$ until $x=a+c+\max\{L_1,L_2\}$, and finally decreases
linearly to $0$ as $x$ goes from $a+c+\max\{L_1,L_2\}$ to $b+d$.
Hence the result is completed. 
\end{proof}

\begin{comment}
Now we present the stability result in one-interval case. 
\begin{thm}
\label{thm: stable_1}
Let $f: \mathbb{R} \to \mathbb{R}$ be defined as the indicator function on the interval $[a_1, b_1]$, i.e., $f(x)=\mathds{1}_{[a_1,b_1]}(x)$, where $a_1, b_1 \in \mathbb{R}$ and $a_1 < b_1$. Consider another function $g: \mathbb{R} \to \mathbb{R}$  defined as $g(x)=\mathds{1}_{[\mytilde{a}_1,\mytilde{b}_1]}(x)$, where $\mytilde{a}_1, \mytilde{b}_1 \in \mathbb{R}$  and $\mytilde{a}_1 < \mytilde{b}_1$. For any given $\epsilon > 0$,  if $\||\widehat{f}|^2-|\widehat{g}|^2\|_2 \leq  \epsilon,$ then 
\begin{equation}
\label{eqn: final_one_interval}
|\mytilde{b}_1-\mytilde{a}_1-(b_1-a_1)|\lesssim \frac{\epsilon}{\sqrt{b_1-a_1}}. 
\end{equation}
\end{thm}
\end{comment}
By Lemma \ref{lem: key}, we now proceed to the proof of Theorem \ref{thm:stable}~$(i)$.
\begin{proof}[ Proof of Theorem \ref{thm:stable}~(i)]
Assume that $I=[a_1,b_1]$ and $\mytilde{I}=[\mytilde{a}_1,\mytilde{b}_1]$. Hence, it is enough to show that
\begin{equation}
\label{eqn: final_one_interval}
|\mytilde{b}_1-\mytilde{a}_1-(b_1-a_1)|\leq \sqrt{\frac{3}{4\pi}}\cdot  \frac{\epsilon}{\sqrt{b_1-a_1}}. 
\end{equation}
By Definition \ref{def: unique_recovery}, we may assume without loss of generality that $a_1 = \mytilde{a}_1 = 0$ via translation. Consequently, proving \eqref{eqn: final_one_interval} reduces to establishing the following inequality:
\begin{equation}
\label{eqn: one_dimensional_temp}
|\mytilde{b}_1 - b_1| \leq \sqrt{\frac{3}{4\pi}}\cdot \frac{\epsilon}{\sqrt{b_1}}.
\end{equation}
%This reformulation effectively reduces the problem to comparing the right endpoints of the intervals, given that their left endpoints now coincide at the origin.
By the properties of the Fourier transform, for any  $f\in \cI$, we have:
\begin{equation}
\label{eqn: f_property}
\widehat{f\mycircledast f_-}=|\widehat{f}|^2\qquad \text{and}\qquad 2\pi\|f\|_2^2=\|\widehat{f}\|_2^2,
\end{equation}
where $f_-(x):=f(-x)$.
Therefore, the condition $\||\widehat{f}|^2-|\widehat{g}|^2\|_2\leq \epsilon$ is equivalent to 
\begin{equation}
\label{eqn: temp1}
\|\left(\mathds{1}_{[0,b_1]}\mycircledast\mathds{1}_{[-b_1,0]}\right)(x)-\left(\mathds{1}_{[0,\mytilde{b}_1]}\mycircledast\mathds{1}_{[-\mytilde{b}_1,0]}\right)(x)\|_2^2\leq \frac{\epsilon^2}{2\pi}.
\end{equation}
According to (\ref{eqn: ab_cd}), for any given $b>0$, we have:
\[
\left(\mathds{1}_{[0,b]}\mycircledast\mathds{1}_{[-b,0]}\right)(x)=
\begin{cases}
x+b,\ &\text{if}\ x\in [-b,0];\\
b-x,\ &\text{if}\ x\in [0,b];\\
0,\ &\text{otherwise}.
\end{cases}
\]
Let
 \[
b_{\min}:=\min\{b_1,\mytilde{b}_1\}\quad \text{and} \quad b_{\max}:=\max\{b_1,\mytilde{b}_1\}.
\]
 We can directly obtain:  
\begin{equation}
\label{eqn: temp2}
\begin{aligned}
&\|\left(\mathds{1}_{[0,b_1]}\mycircledast\mathds{1}_{[-b_1,0]}\right)(x)-\left(\mathds{1}_{[0,\mytilde{b}_1]}\mycircledast\mathds{1}_{[-\mytilde{b}_1,0]}\right)(x)\|_2^2\\
=&2\int_{0}^{b_{\min}}(b_{\max}-b_{\min})^2dx+2\int_{b_{\min}}^{b_{\max}}(b_{\max}-x)^2dx\\
=& 2\cdot b_{\min}\cdot (b_{\max}-b_{\min})^2+\frac{2}{3}\cdot (b_{\max}-b_{\min})^3\\
=&(b_{\max}-b_{\min})^2\cdot (\frac{2}{3}\cdot b_{\max}+\frac{4}{3}\cdot b_{\min})\geq \frac{2b_1}{3}\cdot (\mytilde{b}_1-b_1)^2.
\end{aligned}
\end{equation}
Substituting (\ref{eqn: temp2}) into (\ref{eqn: temp1}), it leads to:
\[
|\mytilde{b}_{1}-b_{1}|\,\,\leq\,\, \sqrt{\frac{3}{4\pi}}\cdot \frac{\epsilon}{ \sqrt{b_1}},
\]
which leads to the conclusion in (\ref{eqn: one_dimensional_temp}).
\end{proof}

\begin{comment}
We next consider part (ii) of Theorem \ref{thm:stable}, which concerns the two-interval case.
Assume that
\[
f(x) = \mathds{1}_{[a_1,b_1]}(x) + \mathds{1}_{[a_2,b_2]}(x),
\]
where $a_1<b_1<a_2<b_2$.
We begin by deriving the following fundamental formulation for $f$. As in the one-interval case, by (\ref{eqn: ab_cd}), we have
\begin{equation}
\label{eqn: f_formulation_origin}
\begin{aligned}
f(x)\mycircledast f(-x)
&=(\mathds{1}_{[a_1,b_1]}(x)+\mathds{1}_{[a_2,b_2]}(x))\mycircledast(\mathds{1}_{[-b_1,-a_1]}(x)+\mathds{1}_{[-b_2,-a_2]}(x))\\
&=h_1(x)+h_2(x)+h_3(x)+h_3(-x).
\end{aligned}
\end{equation}
Here $h_1$, $h_2$ and $h_3$ are defined as:
\[
h_1(x):=
\mathds{1}_{[a_1,b_1]}(x)\mycircledast\mathds{1}_{[-b_1,-a_1]}(x)=
\begin{cases}
x+b_1-a_1,\ &\text{if}\ x\in [-(b_1-a_1),0];\\
b_1-a_1-x,\ &\text{if}\ x\in [0,b_1-a_1];\\
0,\ &\text{otherwise},
\end{cases}
\]
\[
h_2(x):=
\mathds{1}_{[a_2,b_2]}(x)\mycircledast\mathds{1}_{[-b_2,-a_2]}(x)=
\begin{cases}
x+b_2-a_2,\ &\text{if}\ x\in [-(b_2-a_2),0];\\
b_2-a_2-x,\ &\text{if}\ x\in [0,b_2-a_2];\\
0,\ &\text{otherwise},
\end{cases}
\]
and
\[
h_3(x):=
\mathds{1}_{[a_2,b_2]}(x)\mycircledast\mathds{1}_{[-b_1,-a_1]}(x)=
\begin{cases}
x-(a_2-b_1),\ &\text{if}\ x\in [a_2-b_1,a_2-b_1+L_{\min}];\\
L_{\min},\ &\text{if}\ x\in [a_2-b_1+L_{\min},a_2-b_1+L_{\max}]; \\
b_2-a_1-x,\ &\text{if}\ x\in [a_2-b_1+L_{\max},b_2-a_1];\\
0,\ &\text{otherwise},
\end{cases}
\]
with $L_{\max}=\max\{b_2-a_2,b_1-a_1\}$ and $L_{\min}=\min\{b_2-a_2,b_1-a_1\}$. 
\end{comment}
Then we turn to part $(ii)$ of Theorem \ref{thm:stable}, namely the two-interval case.
\begin{comment}
and begin with the following lemma.
It can be proved by (\ref{eqn: ab_cd}) directly.
\begin{lem}
Assume that
\[
f(x) = \mathds{1}_{I_1}(x) + \mathds{1}_{I_2}(x),
\]
where $I_1=[a_1,b_1]$ and $I_2=[a_2,b_2]$ with  $a_1<b_1<a_2<b_2$.
Then
\begin{equation}\label{eqn: f_formulation_origin}
f(x)\mycircledast f(-x)
=
(\abs{I_1}-x)_++(\abs{I_2}-x)_+ +h(x)+h(-x),
\end{equation}
where,
\[
h(x)=\mathds{1}_{[a_2,b_2]}\mycircledast\mathds{1}_{[-b_1,-a_1]}(x)=
\begin{cases}
x-d, & x\in [d,d+I_{\min}],\\
I_{\min}, & x\in [d+I_{\min},d+I_{\max}],\\
d+\abs{I_1}+\abs{I_2}-x, & x\in [d+I_{\max},d+\abs{I_1}+\abs{I_2}],\\
0, & \text{otherwise},
\end{cases}
\]
with 
\[
I_{\min}:=\min\{\abs{I_1},\abs{I_2}\},\qquad
I_{\max}:=\max\{\abs{I_1},\abs{I_2}\},
\]
and $d:=\dist(I_1,I_2)=a_2-b_1$.

\end{lem}

\end{comment}

\begin{comment}
Theorem \ref{thm: stable_2} below establishes that the stability result for two-interval indicator functions $f$ and $g$: if $\||\widehat{f}|^2-|\widehat{g}|^2\|_2\leq \epsilon$ for sufficiently small $\epsilon$, then the differences in their parameters-namely, the shortest interval length $L_{\min}$, the longest interval length $L_{\max}$, and the interval spacing $\min\{|a_2-b_1|,|b_2-a_1|\}$-are all bounded by $\epsilon$. 
\end{comment}

\begin{proof}[Proof of Theorem \ref{thm:stable}~(ii)]
Assume that $I_1=[a_1,b_1], I_2=[a_2,b_2]$ and $\mytilde{I}_1=[\mytilde{a}_1,\mytilde{b}_1], \mytilde{I}_2=[\mytilde{a}_2, \mytilde{b}_2]$. Hence,
it is enough to show that
\begin{equation}
\label{eqn: final_two_interval}
\begin{aligned}
\abs{\min\{b_2-a_2,b_1-a_1\}-\min\{\mytilde{b}_2-\mytilde{a}_2,\mytilde{b}_1-\mytilde{a}_1\}}\lesssim &\  \frac{\epsilon}{\min\{c_0,c\}},\\
\abs{\max\{b_2-a_2,b_1-a_1\}-\max\{\mytilde{b}_2-\mytilde{a}_2,\mytilde{b}_1-\mytilde{a}_1\}}\lesssim &\  \frac{\epsilon}{\min\{c_0,c\}},\\
\abs{ \min\{\abs{a_2-b_1},\abs{b_2-a_1}\}- \min\{\abs{\mytilde{a}_2-\mytilde{b}_1},\abs{\mytilde{b}_2-\mytilde{a}_1}\}}\lesssim &\  \frac{\epsilon}{\min\{c_0,c\}}.
\end{aligned}
\end{equation}
By Definition \ref{def: unique_recovery}, we may assume without loss of generality that $a_1 = \mytilde{a}_1 = 0$ and that
\[
b_1<a_2<b_2, \quad b_2-a_2\geq b_1,
\]
and
\[
\mytilde{b}_1<\mytilde{a}_2<\mytilde{b}_2,\quad \mytilde{b}_2-\mytilde{a}_2\geq \mytilde{b}_1.
\] 
Under these assumptions, the proof of (\ref{eqn: final_two_interval}) can be reduced to demonstrating the following set of inequalities for end-points:
\begin{equation}
\label{eqn: new_est}
|b_1-\mytilde{b}_1|\lesssim\frac{\epsilon}{\min\{c_0,c\}},
\quad |b_2-\mytilde{b}_2|\lesssim \frac{\epsilon}{\min\{c_0,c\}},\quad |a_2-\mytilde{a}_2|\lesssim \frac{\epsilon}{\min\{c_0,c\}}.
\end{equation}
According to the fact in (\ref{eqn: f_property}), we can conclude that:
\[
2\pi\|f\mycircledast f_--g\mycircledast g_-\|_2^2=\||\widehat{f}|-|\widehat{g}|\|_2^2\leq \epsilon^2,
\]
which implies
\begin{equation}
\label{eqn: interval_est}
\int_{\beta_1}^{\beta_2}\left((f\mycircledast {f_-})(x)-(g\mycircledast g_-)(x)\right)^2dx\leq \frac{\epsilon^2}{2\pi},
\end{equation}
 for any $\beta_1,\beta_2\in \mathbb{R}$ such that $\beta_1<\beta_2$. Here, $f_-(x):=f(-x)$ and $g_-(x):=g(-x)$.

The proof is based on applying \((\ref{eqn: interval_est})\) with different
choices of \(\beta_1\) and \(\beta_2\). Combining the resulting estimates yields
\((\ref{eqn: new_est})\).

To use \((\ref{eqn: interval_est})\), we need to rewrite the formulations of $f\mycircledast f_-$ and $g\mycircledast g_-$. 
According to (\ref{eqn: ab_cd}), a simple calculation shows that
\begin{equation}
\label{eqn: f_formulation}
\begin{aligned}
(f\mycircledast f_-)(x)=&\left((\mathds{1}_{[0,b_1]}+\mathds{1}_{[a_2,b_2]})\mycircledast(\mathds{1}_{[-b_1,0]}+\mathds{1}_{[-b_2,-a_2]})\right)(x)\\
=&s_1(x)+s_2(x)+s_3(x)+s_3(-x).
\end{aligned}
\end{equation}
Here $s_1$, $s_2$ and $s_3$ are defined as:
\[
s_1(x):=
\left(\mathds{1}_{[0,b_1]}\mycircledast\mathds{1}_{[-b_1,0]}\right)(x)=
\begin{cases}
x+b_1,\ &\text{if}\ x\in [-b_1,0];\\
b_1-x,\ &\text{if}\ x\in [0,b_1];\\
0,\ &\text{otherwise},
\end{cases}
\]
\[
s_2(x):=
\left(\mathds{1}_{[a_2,b_2]}\mycircledast\mathds{1}_{[-b_2,-a_2]}\right)(x)=
\begin{cases}
x+b_2-a_2,\ &\text{if}\ x\in [-(b_2-a_2),0];\\
b_2-a_2-x,\ &\text{if}\ x\in [0,b_2-a_2];\\
0,\ &\text{otherwise},
\end{cases}
\]
and 
\[
s_3(x):=
\left(\mathds{1}_{[a_2,b_2]}\mycircledast\mathds{1}_{[-b_1,0]}\right)(x)=
\begin{cases}
x-(a_2-b_1),\ &\text{if}\ x\in [a_2-b_1,a_2];\\
b_1,\ &\text{if}\ x\in [a_2,b_2-b_1]; \\
b_2-x,\ &\text{if}\ x\in [b_2-b_1,b_2];\\
0,\ &\text{otherwise}.
\end{cases}
\]
Similarly, we have
\begin{equation}
\label{eqn: g_formulation}
(g\mycircledast g_-)(x)={h}_1(x)+{h}_2(x)+{h}_3(x)+{h}_3(-x), 
\end{equation}
where 
\[
h_1(x):=
\begin{cases}
x+\mytilde{b}_1,\ &\text{if}\ x\in [-\mytilde{b}_1,0];\\
\mytilde{b}_1-x,\ &\text{if}\ x\in [0,\mytilde{b}_1];\\
0,\ &\text{otherwise},
\end{cases}
\quad 
h_2(x):=
\begin{cases}
x+\mytilde{b}_2-\mytilde{a}_2,\ &\text{if}\ x\in [-(\mytilde{b}_2-\mytilde{a}_2),0];\\
\mytilde{b}_2-\mytilde{a}_2-x,\ &\text{if}\ x\in [0,\mytilde{b}_2-\mytilde{a}_2];\\
0,\ &\text{otherwise},
\end{cases}
\]
and 
\[
h_3(x):=
\begin{cases}
x-(\mytilde{a}_2-\mytilde{b}_1),\ &\text{if}\ x\in [\mytilde{a}_2-\mytilde{b}_1,\mytilde{a}_2];\\
\mytilde{b}_1,\ &\text{if}\ x\in [\mytilde{a}_2,\mytilde{b}_2-\mytilde{b}_1]; \\
\mytilde{b}_2-x,\ &\text{if}\ x\in [\mytilde{b}_2-\mytilde{b}_1,\mytilde{b}_2];\\
0,\ &\text{otherwise}.
\end{cases}
\]
The subsequent proof is delineated into four steps:
 
\textbf{Step 1: Prove that $|b_1+b_2-a_2-(\mytilde{b}_1+\mytilde{b}_2-\mytilde{a}_2)|\leq \frac{\epsilon}{\sqrt{2\pi\cdot \min\{c_0,c\}}}$.}

We can apply the formulations in (\ref{eqn: f_formulation}) and (\ref{eqn: g_formulation}). For any $x\in [0,\min\{c,c_0\}]$, we observe that
\begin{equation}
\label{eqn: step1_temp}
(f\mycircledast f_-)(x)-(g\mycircledast g_-)(x)=b_1+b_2-a_2-(\mytilde{b}_1+\mytilde{b}_2-\mytilde{a}_2).
\end{equation}
Here, we use
 \[
 c\leq b_1\leq b_2-a_2,\quad  c\leq \mytilde{b}_1\leq \mytilde{b}_2-\mytilde{a}_2,
 \]
 and 
 \[
 c_0\leq a_2-b_1,\quad c_0\leq \mytilde{a}_2-\mytilde{b}_1.
 \]
Consequently, combining (\ref{eqn: step1_temp}) with (\ref{eqn: interval_est}) for $\beta_1:=0$ and $\beta_2:=\min\{c,c_0\}$, we obtain:
\[
\int_{0}^{\min\{c,c_0\}}(b_1+b_2-a_2-(\mytilde{b}_1+\mytilde{b}_2-\mytilde{a}_2))^2dx\leq \frac{\epsilon^2}{2\pi},
\]
which directly leads to 
\begin{equation}
\label{eqn: result1}
|b_1+b_2-a_2-(\mytilde{b}_1+\mytilde{b}_2-\mytilde{a}_2)|\leq \frac{\epsilon}{\sqrt{2\pi\cdot \min\{c_0,c\}}}.
\end{equation}

\textbf{Step 2: Prove that $|b_2-\mytilde{b}_2|\lesssim \frac{\epsilon}{\min\{c_0,c\}}$.}

We First claim:
\begin{equation}\label{eq:b22}
|b_2-\mytilde{b}_2|\leq c/2.
\end{equation}
 Given that $|b_2-\mytilde{b}_2|\leq c/2$ and $c\leq \min\{b_1,\mytilde{b}_1\}$, we can deduce:
\[
x\geq \min\{b_2,\mytilde{b}_2\}-c/2\geq \max\{b_2,\mytilde{b}_2\}-c,
\]
for any $x\in [\min\{b_2,\mytilde{b}_2\}-c/2,\min\{b_2,\mytilde{b}_2\}]$. It implies that, for $x\in [\min\{b_2,\mytilde{b}_2\}-c/2,\min\{b_2,\mytilde{b}_2\}]$, 
\[
x\geq b_2-c\geq b_2-b_1\geq b_2-a_2\geq b_1,
\]
and 
\[
x\geq \mytilde{b}_2-c\geq \mytilde{b}_2-\mytilde{b}_1\geq \mytilde{b}_2-\mytilde{a}_2\geq \mytilde{b}_1,
\]
based on the fact that
\[
c\leq \min\{b_1,\mytilde{b}_1\},\quad b_1\leq b_2-a_2\quad \text{and}\quad \mytilde{b}_1\leq \mytilde{b}_2-\mytilde{a}_2.
\]
 Consequently, based on  (\ref{eqn: f_formulation}) and (\ref{eqn: g_formulation}), for $x\in [\min\{b_2,\mytilde{b}_2\}-c/2,\min\{b_2,\mytilde{b}_2\}]$, we have:
\[
(f\mycircledast f_-)(x)-(g\mycircledast g_-)(x)=b_2-x-(\mytilde{b}_2-x)=b_2-\mytilde{b}_2,
\]
and hence we use  (\ref{eqn: interval_est}) to obtain
\[
\frac{c}{2}\cdot (b_2-\mytilde{b}_2)^2=\int_{\min\{b_2,\mytilde{b}_2\}-c/2}^{\min\{b_2,\mytilde{b}_2\}} ((f\mycircledast f_-)(x)-(g\mycircledast g_-)(x))^2dx\leq \frac{\epsilon^2}{2\pi},
\]
which implies
\begin{equation}
\label{eqn: b_2 and b_2'}
|b_2-\mytilde{b}_2|\,\,\leq\,\, \frac{\epsilon}{\sqrt{\pi \cdot c}}\lesssim \frac{\epsilon}{\min\{c_0,c\}}.
\end{equation}

It remains to prove (\ref{eq:b22}). Let us proceed by contradiction. Suppose that $|b_2-\mytilde{b}_2|> c/2$. Without loss of generality, assume $\mytilde{b}_2<b_2-c/2$. Applying the formulations in (\ref{eqn: f_formulation}) and (\ref{eqn: g_formulation}) and noting that $b_2-b_1<b_2-c/2$, for any $x\in [b_2-c/2,b_2]$,  we obtain:
\[
(f\mycircledast f_-)(x)-(g\mycircledast g_-)(x)=b_2-x.
\]
Invoking (\ref{eqn: interval_est}) with $\beta_1=b_2-c/2$ and $\beta_2=b_2$ yields:
\[
 \frac{c^3}{24} = \int_{b_2-c/2}^{b_2} ((f\mycircledast f_-)(x)-(g\mycircledast g_-)(x))^2dx\leq \frac{\epsilon^2}{2\pi},
\]
which contradicts with the condition in (\ref{eqn: epsilon}).

\textbf{Step 3: Prove that $|b_1-\mytilde{b}_1|\lesssim \frac{\epsilon}{\min\{c_0,c\}}$.}

First of all, we proceed to demonstrate that 
\[
|\mytilde{b}_1-b_1|\leq c_0/4.
\]
 Let us assume, contrary to our claim, that $|\mytilde{b}_1-b_1|> c_0/4$. Without loss of generality, suppose $\mytilde{b}_1<b_1-c_0/4$. Combining this with (\ref{eqn: b_2 and b_2'}) and the constraint on $\epsilon$ in (\ref{eqn: epsilon}), we derive:
\begin{equation}
\label{eqn: step3_temp}
\mytilde{b}_2-\mytilde{b}_1>\mytilde{b}_2-b_1+c_0/4\geq b_2-b_1+c_0/4- \frac{\epsilon}{\sqrt{2\pi \cdot c}}\geq b_2-b_1+c_0/6.
\end{equation}
For any $x\in \left[\mytilde{b}_2-\mytilde{b}_1-\frac{c_0}{6}, \mytilde{b}_2-\mytilde{b}_1-\frac{c_0}{12}\right]$, we can establish that:
\begin{equation}\label{eq:xfan1}
 x\geq \mytilde{b}_2-\mytilde{b}_1-c_0\overset{(a)}\geq \mytilde{b}_2-\mytilde{a}_2\geq \mytilde{b}_1,\quad x\leq \mytilde{b}_2-\mytilde{b}_1,
 \end{equation}
 and
\begin{equation}\label{eq:xfan2}
\quad x\geq \mytilde{b}_2-\mytilde{b}_1-\frac{c_0}{6}\overset{(b)}\geq b_2-b_1\geq b_2-a_2\geq b_1.
\end{equation}
Here $(a)$ follows from $c_0
 \leq  \dist(\mytilde{I_1}, \mytilde{I_2})=\mytilde{a}_2-\mytilde{b}_1$, and $(b)$ is based on (\ref{eqn: step3_temp}). 
Using the range of \(x\) given by (\ref{eq:xfan1}) and (\ref{eq:xfan2}), together with the piecewise formulas in
(\ref{eqn: f_formulation}) and (\ref{eqn: g_formulation}), we obtain, for
\[
x\in \left[\mytilde{b}_2-\mytilde{b}_1-\frac{c_0}{6}, \mytilde{b}_2-\mytilde{b}_1-\frac{c_0}{12}\right],
\]
 that
\[
(f\mycircledast f_-)(x)-(g\mycircledast g_-)(x)=b_2-x-\mytilde{b}_1\geq b_2-\mytilde{b}_2+\frac{c_0}{12}\geq \frac{c_0}{24}.
\]
The last inequality follows from (\ref{eqn: b_2 and b_2'}) together with the assumption on \(\epsilon\) in (\ref{eqn: epsilon}), namely,
\[
|b_2-\mytilde{b}_2|\,\,\leq\,\, \frac{\epsilon}{\sqrt{\pi\cdot c}}\,\,\leq\,\, \frac{c_0}{24}.
\]
 Thus invoking (\ref{eqn: interval_est}) with $\beta_1=\mytilde{b}_2-\mytilde{b}_1-\frac{c_0}{6}$ and $\beta_2=\mytilde{b}_2-\mytilde{b}_1-\frac{c_0}{12}$ yields:
\[
\frac{c_0^3}{6912}\leq \int_{\mytilde{b}_2-\mytilde{b}_1-\frac{c_0}{6}}^{\mytilde{b}_2-\mytilde{b}_1-\frac{c_0}{12}} ((f\mycircledast f_-)(x)-(g\mycircledast g_-)(x))^2dx\leq \frac{\epsilon^2}{2\pi},
\]
which contradicts with the condition in (\ref{eqn: epsilon}).

Therefore, we conclude that $|\mytilde{b}_1-b_1|\leq c_0/4$, which implies
\[
b_1+c_0/4\geq \mytilde{b}_1\geq b_1-c_0/4.
\]
For
\[
x\in \left[\mytilde{b}_2-\mytilde{b}_1-\frac{5c_0}{8},
\mytilde{b}_2-\mytilde{b}_1-\frac{c_0}{2}\right],
\]
we have
\[
x\geq \mytilde{b}_2-\mytilde{b}_1-\frac{5c_0}{8}
\geq \mytilde{b}_2-\mytilde{a}_2
\geq \mytilde{b}_1,
\qquad
x\leq \mytilde{b}_2-\mytilde{b}_1,
\]
and
\[
x\geq \mytilde{b}_2-\mytilde{b}_1-\frac{5c_0}{8}
\overset{(c)}{\geq}
b_2-b_1-\frac{15c_0}{16}
\geq b_2-a_2
\geq b_1.
\]
Moreover,
\[
x\leq \mytilde{b}_2-\mytilde{b}_1-\frac{c_0}{2}
\overset{(d)}{\leq}
b_2-b_1-\frac{3c_0}{16}.
\]
Here, \((c)\) and \((d)\) follow from
\[
|\mytilde{b}_1-b_1|\leq \frac{c_0}{4}
\quad \text{and} \quad
|\mytilde{b}_2-b_2|
\leq \frac{\epsilon}{\sqrt{\pi c}}
\leq \frac{c_0}{16}.
\]
Therefore, from the range of $x$ specified above, together with (\ref{eqn: f_formulation}) and (\ref{eqn: g_formulation}), it follows that, 
for $x\in [\mytilde{b}_2-\mytilde{b}_1-5c_0/8,\mytilde{b}_2-\mytilde{b}_1-c_0/2]$, we have 
\[
(f\mycircledast f_-)(x)-(g\mycircledast g_-)(x)=\mytilde{b}_1-b_1.
\]
Applying (\ref{eqn: interval_est}) with $\beta_1=\mytilde{b}_2-\mytilde{b}_1-5c_0/8$ and $\beta_2=\mytilde{b}_2-\mytilde{b}_1-c_0/2$ yields:
\[
\frac{c_0}{8}\cdot (b_1-\mytilde{b}_1)^2=\int_{\mytilde{b}_2-\mytilde{b}_1-5c_0/8}^{\mytilde{b}_2-\mytilde{b}_1-c_0/2} (f\mycircledast f_-)(x)-(g\mycircledast g_-)(x))^2 dx\leq \frac{\epsilon^2}{2\pi},
\]
which implies
\begin{equation}
\label{eqn: b_1 and b_1'}
|b_1-\mytilde{b}_1|\leq \frac{2\epsilon}{\sqrt{\pi \cdot c_0}}\lesssim \frac{\epsilon}{\min\{c_0,c\}}.
\end{equation}
\textbf{Step 4: Prove that $|a_2-\mytilde{a}_2|\lesssim \frac{\epsilon}{\sqrt{\min\{c_0,c\}}}$.}

To conclude our proof, we can leverage the results established in (\ref{eqn: result1}), (\ref{eqn: b_2 and b_2'}), and (\ref{eqn: b_1 and b_1'}). By synthesizing these findings, we can directly infer that:
\[
|a_2-\mytilde{a}_2|\lesssim \frac{\epsilon}{\sqrt{\min\{c_0,c\}}}.
\]
 This final step completes our demonstration of the inequality (\ref{eqn: new_est}), thereby concluding the proof.
\end{proof}

\section{discrete Fourier phase retrieval for real signals}
\label{sec:disF}

In this section, we establish a result on discrete Fourier phase retrieval 
for real-valued signals, which may be of independent interest and will also 
be used to construct counterexamples for Theorem~\ref{th:fanli} in 
Section~\ref{sec: sec5}. We further introduce the \emph{turnpike problem} 
and several related results that will be utilized in Section~\ref{sec: sec6}.

\subsection{Discrete Fourier phase retrieval }

 Let $\mathbf{x} = (x_{0}, x_{1}, \ldots, x_{d-1})^T \in \mathbb{F}^d$, where $\mathbb{F} \in \{\mathbb{R}, \mathbb{C}\}$. Denote the discrete Fourier transforms by $\hat{\mathbf{x}}(\omega)$, defined by
\begin{equation}\label{eq:dftxy}
\hat{\mathbf{x}}(\omega) := \sum_{j=0}^{d-1} x_j \cdot \exp(-2\pi \mathrm{i}  j \omega / d), \qquad \text{for all}\  \omega \in \mathbb{R}.
\end{equation}
Let $X(z)$  denote the $Z$-transform of ${\mathbf x}$, defined by
\begin{equation}
\label{eqn: x_z}
X(z) = \sum_{j=0}^{d-1} x_j \cdot z^j,  \qquad \text{for all}\ z \in \mathbb{C}.
\end{equation}
Then for any $w\in \mathbb{R}$, $\hat{\mathbf{x}}(\omega)=X(\exp(-2\pi \mathrm{i}w/d)) $. The goal of discrete Fourier phase retrieval is to determine $\mathbf{x}$ from magnitude-only measurements of its discrete Fourier transform, that is, from $\abs{\hat{\mathbf{x}}(\omega)}$ for all $\omega \in \mathbb{R}$.

For the case ${\mathbb F} = {\mathbb C}$, the theorem below gives the necessary and sufficient conditions for $\abs{\hat{\mathbf{x}}(\omega)}=\abs{\hat{\mathbf{y}}(\omega)}$ to hold, where $\mathbf{x}, \mathbf{y}\in {\mathbb C}^d$.

\begin{thm}\label{th:review}\cite[Theorem 3.7]{Review}
 Let 
 \[
 \mathbf{x} = (x_0, x_1, \ldots, x_{d-1})^T \in \mathbb{C}^d\quad \text{and}\quad \mathbf{y} = (y_0, y_1, \ldots, y_{d-1})^T \in \mathbb{C}^d.
 \] 
 Let $X(z)$ and $Y(z)$ denote their respective $Z$-transforms, defined by
\[
X(z) = \sum_{j=0}^{d-1} x_j \cdot z^j, \quad Y(z) = \sum_{j=0}^{d-1} y_j \cdot z^j,\qquad \text{for all}\ z \in \mathbb{C}.
\]
Then $|\hat{\mathbf x}(\omega)| = |\hat{\mathbf y}(\omega)|$ for all $\omega \in {\mathbb R}$ if and only if there exist a factorization $Y = Y_1 \cdot Y_2$, a constant $\gamma$ with $|\gamma| = 1$, and $\tau \in \mathbb{Z}$ such that
\begin{equation}\label{eq:XY1Y2}
X(z) = \gamma \cdot z^\tau \cdot Y_1(z) \cdot \overline{Y_2(\overline{z}^{-1})}.
\end{equation}
\end{thm}

Theorem~\ref{th:review} treats the general case of complex vectors. In what follows, we restrict to the real setting, where the additional structure yields a more precise result.
\begin{thm}\label{th:realfac}
Let $\mathbf{x} = (x_0, x_1, \ldots, x_{d-1})^T\in \mathbb{R}^d$ and $\mathbf{y} = (y_0, y_1, \ldots, y_{d-1})^T \in \mathbb{R}^d$.
 Let $X(z), Y(z)$,  $\hat{\mathbf{x}}(\omega)$ and $\hat{\mathbf{y}}(\omega)$ be defined in Theorem \ref{th:review}.
Then $|\hat{\mathbf x}(\omega)| = |\hat{\mathbf y}(\omega)|$ for all $\omega \in {\mathbb R}$ if and only if there exist a real factorization $Y = R_1 \cdot R_2$, 
a constant $\gamma_1\in \{1,-1\}$ and $\tau_1 \in \mathbb{Z}$ such that
\begin{equation}\label{eq:XR1R2}
X(z) =  \gamma_1\cdot z^{\tau_1} \cdot R_1(z) \cdot {R_2(1/{z})},
\end{equation}
where $R_1$  and $R_2$   are real-coefficient polynomials.
\end{thm}

\begin{proof}

We first prove the ``if'' direction: assume that there exist real-coefficient polynomials $R_1$ and $R_2$ such that $Y(z)=R_1(z)\cdot R_2(z)$ and $X(z) = \gamma_1\cdot  z^{\tau_1} \cdot R_1(z) \cdot {R_2(1/{z})}$. For any $\omega \in {\mathbb R}$, we have $\abs{R_2(1/\exp({\ri \omega}))}=\abs{R_2(\exp({-\ri \omega}))}=\abs{\overline{R_2(\exp({\ri \omega}))}}=\abs{R_2(\exp({\ri \omega}))}$. 
Hence,
\[
\begin{aligned}
\abs{X(\exp({\ri \omega}))}&=\abs{\exp({\ri\tau_1 \omega})\cdot R_1(\exp({\ri \omega})) \cdot {R_2(1/{\exp({\ri \omega})})}}\\
&=\abs{R_1(\exp(-{\ri \omega}))}\cdot \abs{R_2(\exp({\ri \omega}))}\\
&=\abs{Y(\exp({\ri\omega}))},
\end{aligned}
\]
which implies that $|\hat{\mathbf x}(\omega)| = |\hat{\mathbf y}(\omega)|$ for all $\omega \in {\mathbb R}$.

Conversely, we now turn to the ``only if'' direction.
Assume that $|\hat{\mathbf x}(\omega)| = |\hat{\mathbf y}(\omega)|$ for all $\omega \in {\mathbb R}$.
According to Theorem \ref{th:review}, 
there exist a factorization $Y = Y_1 \cdot Y_2$, a constant $\gamma$ with $|\gamma| = 1$, and $\tau \in \mathbb{Z}$ such that
\begin{equation}\label{eq:Y12}
X(z) = \gamma \cdot z^\tau \cdot Y_1(z) \cdot \overline{Y_2(\overline{z}^{-1})}.
\end{equation}

Since $\mathbf{x}$ and $\mathbf{y}$ are real-valued vectors, $X(z)$ and $Y(z)$ are polynomials with real coefficients. Our objective is to show that the complex factors $Y_1$ and $Y_2$ can be rearranged into a real-coefficient factorization $Y(z) = R_1(z)\cdot R_2(z)$ that satisfies \eqref{eq:XR1R2}.

Without loss of generality, We  assume that the coefficient of the leading term of $Y_2$ is $1$.
According to the Fundamental Theorem of Algebra, we can write $Y_2$ in the following form
\begin{equation}\label{eq:R2}
Y_2(z)=R_2(z)\cdot \prod_{j=1}^k(z-\rho_j)^{u_j}.
\end{equation}
Here, $R_2$ is a polynomial with real coefficients; $\rho_j$, $j=1,\ldots,k$, 
are the non-real roots of $Y_2$ such that no two of them are complex conjugates 
of each other, i.e., $\rho_{j_1}\neq \overline{\rho_{j_2}}$ for all 
$j_1,j_2=1,\ldots,k$; and $u_j\in \mathbb{Z}_{+}$, $j=1,\ldots,k$, 
denotes the multiplicity of $\rho_j$, respectively.

Take $R_1$ as:
\[
R_1(z)=Y_1(z)\cdot Y_3(z),
\]
 where 
 \[
 Y_3(z):=\prod_{j=1}^k(z-\rho_j)^{u_j}
 \]
  and $R_2$  as defined in (\ref{eq:R2}). Then we have
\[
\begin{aligned}
Y(z)&=Y_1(z)\cdot Y_2(z)=Y_1(z)\cdot R_2(z)\cdot Y_3(z)=Y_1(z)\cdot  Y_3(z)\cdot R_2(z)=R_1(z)\cdot R_2(z).
\end{aligned}
\]
Since both $Y$ and $R_2$   are polynomials with real coefficients,
it follows that $R_1(z)$ must also have real coefficients.

It remains to prove (\ref{eq:XR1R2}). We claim that 
\begin{equation}\label{eq:midY3}
z^n\cdot \overline{Y_3(\overline{z}^{-1})}=\gamma_0\cdot Y_3(z),
\end{equation}
where $n$ is the degree of $Y_3$ and $\gamma_0\in {\mathbb C}$ is a constant satisfying $\abs{\gamma_0}=1$.
 This implies   (\ref{eq:XR1R2}). Indeed,
 according to (\ref{eq:Y12}) and (\ref{eq:midY3}), we have
 \[
 \begin{aligned}
X(z)& = \gamma \cdot z^\tau \cdot Y_1(z) \cdot \overline{Y_2(\overline{z}^{-1})}\\
&=\gamma \cdot z^\tau \cdot Y_1(z) \cdot \overline{R_2(\overline{z}^{-1})} \cdot 
\overline{Y_3(\overline{z}^{-1})}\\
&=
\gamma\cdot z^{\tau-n} \cdot Y_1(z) \cdot {R_2(1/{z})} \cdot 
z^n\cdot \overline{Y_3(\overline{z}^{-1})}\\
&=
\gamma\cdot \gamma_0\cdot z^{\tau-n} \cdot Y_1(z) \cdot {R_2(1/{z})} \cdot Y_3(z)\\
&=\gamma\cdot \gamma_0\cdot z^{\tau-n}\cdot R_1(z) \cdot R_2(1/z)\\
&= \gamma_1\cdot z^{\tau_1}\cdot R_1(z) \cdot R_2(1/z).
\end{aligned}
\]
Here, we take $\gamma_1:=\gamma\cdot \gamma_0$ and $\tau_1:=\tau-n$.
 Since $X$, $R_1$, and $R_2$ are all polynomials with real coefficients, the product $\gamma_1=\gamma \cdot \gamma_0$ must be real. Given that $|\gamma| = |\gamma_0| = 1$, it follows that $\gamma_1 \in \{1, -1\}$.

It remains to establish \eqref{eq:midY3}.
Let $Y_3(z)$ be rewritten as:
\[
Y_3(z)=\prod_{j=1}^k(z-\rho_j)^{u_j}=c_0+c_1\cdot z+\cdots+c_n \cdot z^n,
\]
where $n={\rm deg}(Y_3)$ and $c_n=1$.
Then we have
\[
\overline{Y_3(\overline{z}^{-1})}=\overline{c_0}+\overline{c_1}\cdot z^{-1}+\cdots+\overline{c_n}\cdot z^{-n},
\]
which implies
\[
\widetilde{Y_3}(z):=z^n\cdot \overline{Y_3(\overline{z}^{-1})}
\]
is a polynomial in $z$ with degree $n$.

 Let $Z(Y_3)$ and $Z(\widetilde{Y}_3)$ denote the multisets of roots of $Y_3$ and $\widetilde{Y_3}$, respectively, including multiplicities. 
 We claim that  
 \begin{equation}\label{eq:zy3c0cn}
 Z(Y_3) = Z(\widetilde{Y_3}),\quad \text{and}\quad  \abs{c_0}=1,
 \end{equation}
 which directly implies (\ref{eq:midY3}).

 We now proceed to prove (\ref{eq:zy3c0cn}).
Since $Z(Y_3)$ and $Z(\widetilde{Y_3})$ are multisets of the same cardinality, 
to establish $Z(Y_3) = Z(\widetilde{Y_3})$ it suffices to show that each 
$\rho \in Z(Y_3)$ belongs to $Z(\widetilde{Y_3})$ with the same multiplicity.

For clarity, we present the proof assuming that
each root $\rho_j$  of $Y_3$   is simple, i.e., $u_j=1, j=1,\ldots,k$. Extending the result to the general case is straightforward.

First, if there exists a root $\rho \in Z(Y_3)$ such that $|\rho| = 1$, then $\overline{\rho} = \rho^{-1}$. 
 It follows that $Y_3(\overline{\rho}^{-1}) = Y_3(\rho) = 0$, which implies $\conj{Y_3(\overline{\rho}^{-1})}=0$.
  Consequently, $\widetilde{Y_3}(\rho) = \rho^n\cdot \overline{Y_3(\overline{\rho}^{-1})} = 0$, and thus $\rho \in Z(\widetilde{Y_3})$.

Next, consider the case $\rho \in Z(Y_3)$ with $0<\abs{\rho}\neq 1$. We claim that 
\begin{equation}
\label{eqn: claim_temp}
\{\overline{\rho},\rho^{-1}\}\subset Z(Y_1)\quad \text{and} \quad \{\overline{\rho}^{-1},\rho\}\subset Z(Y_3).
\end{equation}
Thus by direct calculation, the constant term $c_0$ of $Y_3$ satisfies $\abs{c_0}=1$. According to the definition of  $\widetilde{Y_3}(z)$,  
 the fact that $\rho, \overline{\rho}^{-1}\in Z(Y_3)$ implies that $\overline{\rho}^{-1},\rho\in Z(\widetilde{Y_3})$.  Therefore, \eqref{eq:zy3c0cn} follows. 

The only thing left is to prove \eqref{eqn: claim_temp}. First of all, we show that 
$\rho \in Z(Y_3)$ implies $\overline{\rho}, \rho^{-1} \in Z(Y_1)$. 
The well-known fact that the non-real roots of a real polynomial occur in 
conjugate pairs plays a key role in our argument. Since $Y/R_2 = Y_1 \cdot Y_3$ 
is a polynomial with real coefficients and $\overline{\rho} \notin Z(Y_3)$ by 
construction, we must have $\overline{\rho} \in Z(Y_1)$. Since 
$\rho \in Z(Y_3)$, it follows directly that $\overline{\rho}^{-1} \in 
Z(\widetilde{Y_3})$.
Furthermore, since $X(z)$ in \eqref{eq:Y12} has real coefficients, the roots 
of the product $Y_1 \cdot \widetilde{Y_3}$ must occur in conjugate pairs. Hence 
the conjugate $\rho^{-1}$ of the root $\overline{\rho}^{-1} \in Z(\widetilde{Y_3})$ 
must belong to $Z(Y_1) \cup Z(\widetilde{Y_3})$. We conclude that 
$\rho^{-1} \in Z(Y_1)$, since $\rho^{-1} \in Z(\widetilde{Y_3})$ would imply 
$\overline{\rho} \in Z(Y_3)$, contradicting $\overline{\rho} \notin Z(Y_3)$.
 
Then we prove $\overline{\rho}^{-1}\in Z(Y_3)$, which complete the proof of \eqref{eqn: claim_temp}. 
%Specifically, the multiplicity of $\rho^{-1}$ as a root of $Y_1$ is strictly less than that of its conjugate, $\overline{\rho}^{-1}$.  may not be correct!
Following from the fact that $\overline{\rho}^{-1} \in Z(\widetilde{Y_3})$ and the product $\gamma\cdot  Y_1\cdot  \widetilde{Y_3}$ has real coefficients, these conditions imply that $\rho^{-1}$ is a root of $Y_1$ whose conjugate $\overline{\rho}^{-1}$ resides in $Z(\widetilde{Y_3})$ rather than $Z(Y_1)$.
 To restore the conjugate symmetry required for the real-coefficient polynomial $Y_1\cdot Y_3$, this deficit must be compensated by including $\overline{\rho}^{-1}$ in $Z(Y_3)$.

\end{proof}

%\textcolor{red}
{Let $\mathbf{x} = (x_{0}, x_{1}, \ldots, x_{d-1})^T \in \mathbb{R}^d$ 
and $\mathbf{y} = (y_{0}, y_{1}, \ldots, y_{d-1})^T \in \mathbb{R}^d$. 
In analogy with Definition~\ref{def: unique_recovery} and based on Proposition~3.3 
in~\cite{Review}, we restrict to the real case and say that $\mathbf{x}$ and $\mathbf{y}$ 
are \emph{equivalent}, written $\mathbf{x} \sim \mathbf{y}$, if
\begin{equation}
    \mathbf{y} = \gamma \cdot T_{\tau} \cdot \mathbf{x} \quad \text{or} \quad 
    \mathbf{y} = \gamma \cdot T_{\tau} \cdot R\mathbf{x}
\end{equation}
for some $\gamma \in \{-1, +1\}$ and $\tau \in \mathbb{Z}$, where the reflection 
operator $R$ and the shift operator $T_{\tau}$ are defined by
$$
(R\mathbf{x})_j = x_{-j \bmod d}, \qquad (T_{\tau}\mathbf{x})_j = x_{(j-\tau) \bmod d},
\qquad j \in \{0, \ldots, d-1\}.
$$
 \\
\indent According to Theorem \ref{th:realfac}, we can directly get the following corollary:
\begin{coro}\label{co:ine}
Let $\mathbf{x} = (x_0, x_1, \ldots, x_{d-1})^T \in \mathbb{R}^d$ 
and $\mathbf{y} = (y_0, y_1, \ldots, y_{d-1})^T \in \mathbb{R}^d$. 
Let $X(z)$, $Y(z)$, $\hat{\mathbf{x}}(\omega)$, and 
$\hat{\mathbf{y}}(\omega)$ be as defined in Theorem~\ref{th:review}. 
Suppose that $|\hat{\mathbf{x}}(\omega)| = |\hat{\mathbf{y}}(\omega)|$ 
for all $\omega \in \mathbb{R}$, and write $Y(z) = R_1(z) \cdot R_2(z)$ 
and
\begin{equation}\label{eq:XR1R2_new}
    X(z) = \gamma \cdot z^{\tau} \cdot R_1(z) \cdot R_2(1/z),
\end{equation}
for some $\gamma \in \{-1, +1\}$ and $\tau \in \mathbb{Z}$, where 
$R_1(z)$ and $R_2(z)$ are polynomials with real coefficients.
Then $\mathbf{x}$ and $\mathbf{y}$ are inequivalent ($\mathbf{x}\not\sim\mathbf{y}$) if and only if 
\[
1\le \deg(R_1)<\deg(Y) \quad \text{with}\quad \gamma_1\cdot z^{\tau_1}\cdot  R_1(1/z)\not\equiv R_1(z) 
\]
for any integer $\tau_1$ and $\gamma_1\in \{-1,+1\}$, and
\[
\gamma_2\cdot z^{\tau_2}\cdot  R_2(1/z)\not\equiv R_2(z) 
\]
for any integer $\tau_2$ and $\gamma_2\in \{-1,+1\}$.
\end{coro}}
%The original result not correct! 

\subsection{The turnpike problem}
\label{sec: turnpike}
The \emph{turnpike problem} \cite{HD} asks whether a point set
$$
P = \{a_1, \ldots, a_m\} \subset \mathbb{R}
$$
can be recovered, up to translation and reflection, from its iset of 
pairwise differences
$$
P - P = \{a_i - a_j : 1 \leq i,j \leq m\}.
$$
The following result shows that, if $P\subset \mathbb{Z}$ satisfies a
\emph{collision-free condition}, then $P$ is uniquely determined by $P-P$,
up to translation and reflection. 
\begin{defn}[Collision-Free Condition]
\label{def: collision_free}
A finite set $P = \{\va_1, \ldots, \va_m\} \subset \R^d$ satisfies the \emph{collision-free condition} if for any indices $i_1, i_2, j_1, j_2 \in \{1, \ldots, m\}$ with $i_1 \neq i_2$, the equality 
\[
\va_{i_1} - \va_{i_2} = \va_{j_1} - \va_{j_2}
\]
implies $(i_1, i_2) = (j_1, j_2)$.
\end{defn}
The collision-free condition also means that
distinct ordered pairs give rise to distinct differences, namely,
\[
a_{i_1}-a_{i_2}\neq a_{j_1}-a_{j_2},
\]
whenever $(i_1,i_2)\neq (j_1,j_2)$ and $i_1\neq i_2$.

\begin{lem}(\cite{sixpoint,Dan})
\label{lem: separation_cond}
Let $m\neq 6$ be a positive integer, and let
$
P:=\{a_1,\ldots,a_m\}\subset \mathbb{Z}
$.
Assume that $P$ satisfies the collision-free condition.
 If there exists
$
Q:=\{b_1,\ldots,b_m\}\subset \mathbb{Z}
$
such that $Q-Q=P-P$, then
\[
Q=P+z_0\quad \mathrm{or}\quad Q=-P+z_0
\]
for some $z_0\in \mathbb{Z}$.
\end{lem}

\begin{rem}
The turnpike problem is closely related to discrete Fourier phase retrieval. Let
 {\[
{\mathbf p}:=(p_{0},\ldots,p_{d-1})\in \mathbb{R}^d,\qquad
{\mathbf q}:=(q_{0},\ldots,q_{d-1})\in \mathbb{R}^d,
\]
where
\[
p_{j-\min(P\cup Q)}=
\begin{cases}
1, & \text{if } j\in P,\\
0, & \text{if } j\notin P,
\end{cases}
\quad \text{and}\quad
q_{j-\min(P\cup Q)}=
\begin{cases}
1, & \text{if } j\in Q,\\
0, & \text{if } j\notin Q,
\end{cases}
\]
with $d=\max(P\cup Q)-\min(P\cup Q)$}. A simple observation is that
\[
Q-Q=P-P
\]
if and only if
\[
\abs{\widehat{\mathbf p}(\omega)}^2
=
\abs{\widehat{\mathbf q}(\omega)}^2.
\]
By Theorem \ref{th:realfac}, it suffices to consider real factor 
decompositions. The proof in \cite{sixpoint} proceeds in this setting, 
with Theorem \ref{th:realfac} providing the justification for the reduction.

%By Theorem \ref{th:realfac}, it suffices to consider the real factor decomposition. 
%The proof in \cite{sixpoint} is carried out in this real factor case, while
%the treatment of possible complex factors is left implicit. Theorem
%\ref{th:realfac} justifies this reduction, thereby complementing the argument
%and making the underlying reasoning explicit. 有一些重复说明。
\end{rem}

\begin{rem}
It is shown in \cite{sixpoint} and \cite[Page 7, before Example 5.5]{Dan}
that, when $m = 6$, the point set $P$ is not uniquely determined up to 
translation and reflection by the pairwise difference set $P - P$ if and 
only if $P$ takes one of the following forms:
$$
P = \{0,\, a,\, b-2a,\, 2b-2a,\, 2b,\, 3b-a\},
$$
or
$$
P = \{0,\, a,\, 2a+b,\, a+2b,\, 2b-a,\, 3b-a\},
$$
where $a, b \in \mathbb{R}$.

\end{rem}

In this section, we extend Lemma~\ref{lem: separation_cond}, originally established 
for $\mathbb{Z}$, to the real domain $\mathbb{R}$ under the collision-free assumption.
We begin with an auxiliary lemma that generalizes the one-dimensional result of 
Lemma~\ref{lem: separation_cond} to the higher-dimensional setting.

\begin{lem} \label{lem:hig}
Assume that $U,V\subset {\mathbb Q}^d$ are finite. Assume that $U$ and $V$ satisfies the collision-free condition described in Definition \ref{def: collision_free}  with $\abs{U}\neq 6$. If $U-U=V-V$, then 
\begin{equation}\label{eq:deng}
U=V+{\mathbf z}_0, \quad \text{or}\quad U=-V+{\mathbf z}_0,
\end{equation}
for some ${\mathbf z}_0\in {\mathbb Q}^d$.

\end{lem}
\begin{proof}
It suffices to consider the integer case. For any given $U, V \subset \mathbb{Z}^d$ with $|U| \neq 6$, assume that they satisfy the collision-free condition.
Denote $\Delta(X) := X - X$ for any set $X \subset \mathbb{Z}^d$. 
Assume furthermore that $\Delta(U) = \Delta(V)$.
Since $U$, $V$, and $\Delta(U)$ are finite subsets of $\mathbb{Z}^d$, there exists 
a nonzero vector $\mathbf{k}_0 \in \mathbb{Z}^d$ such that

$$
\langle \mathbf{k}_0, \mathbf{x} - \mathbf{y} \rangle \neq 0, 
\quad \text{for all distinct } \mathbf{x}, \mathbf{y} \in U \cup V \cup \Delta(U).
$$

For any given $X\in \mathbb{Z}^d$, set 
\[
\mathbf{k}_0\cdot X:=\{\innerp{\mathbf{k}_0, \mathbf{x}}: \mathbf{x}\in X\}.
\]
and for $\epsilon\geq 0$, set
\begin{equation}
\begin{aligned}
S(\epsilon)&:=\left\{ \vx\in \R^d\ :\  \left\|\vx-\frac{\vk_0}{\|\vk_0\|_2}\right\|_2\leq \epsilon,\  \|\vx\|_2=1 \right\},\\
K(\epsilon)&:=
\left\{\vk\in \Z^d\setminus \{0\}\ :\  \left\|\frac{\vk_0}{\|\vk_0\|_2}- \frac{\vk}{\|\vk\|_2}\right\|_2\leq \epsilon\right\}.
\end{aligned}
\end{equation}
Then it is directly to see that there exists $\epsilon_0>0$  such that for every $\vk \in K(\epsilon_0)$:
\begin{equation}\label{eq:vk}
\innerp{\vk,\vx-\vy}\neq 0,  \quad \text{ for any distinct }  \vx,\vy\in U\cup V\cup \Delta(U).
\end{equation}
For convenience, set $S_0:=S(\epsilon_0)$ and $K_0:=K(\epsilon_0)$.

% Hence, for any $\vk\in K_0$,  both $\vk\cdot U_0\subset  \Z$ and $\vk\cdot V_0\subset \Z$ satisfy the collision-free condition  which satisfy
%\begin{equation}\label{eq:UVdeng}
%\Delta(\vk\cdot U_0)=\vk\cdot \Delta(U_0)=\vk\cdot \Delta(V_0)=\Delta(\vk\cdot V_0).
%\end{equation}

By the stereographic projection, rational points are dense on the unit sphere $\mathbb{S}^{d-1}$.
A simple observation is that each rational point in $S_0$ corresponds to a point in
$$
K_0' := \left\{ \frac{\mathbf{k}}{\|\mathbf{k}\|} : \mathbf{k} \in K_0 \right\},
$$
which implies that $K_0'$ is dense in $S_0$, and hence $\cl K_0' =S_0$, where $\cl K_0'$ denotes the closure of $K_0'$.
Assume that $U = \{\mathbf{u}_1, \ldots, \mathbf{u}_m\} \subset \mathbb{Z}^d$ and 
$V = \{\mathbf{v}_1, \ldots, \mathbf{v}_m\} \subset \mathbb{Z}^d$.
Let $\mathrm{Sym}(m)$ denote the permutation group of degree $m$.
Then there exists $K_1 \subset K_0$ such that the Lebesgue outer measure of 
$\cl {\left\{ \frac{\mathbf{k}}{\|\mathbf{k}\|} : \mathbf{k} \in K_1 \right\}}$ 
on $\mathbb{S}^{d-1}$ is positive, and for any $\mathbf{k} \in K_1$,

$$
\langle \mathbf{k}, \mathbf{u}_j \rangle = s_{\mathbf{k}} \cdot \langle \mathbf{k}, \mathbf{v}_{\pi_{\mathbf{k}}(j)} \rangle + z_{\mathbf{k}}, 
\quad j = 1, \ldots, m,
$$
where $s_{\mathbf{k}} \in \{1, -1\}$, $z_{\mathbf{k}} \in \mathbb{Z}$, and 
$\pi_{\mathbf{k}} \in \mathrm{Sym}(m)$ all depend on $\mathbf{k}$.
This follows from Lemma~\ref{lem: separation_cond} and the construction of $\mathbf{k}$ 
satisfying \eqref{eq:vk}, which ensures the collision-free condition for 
$\mathbf{k} \cdot U$ and $\mathbf{k} \cdot V$, together with
\begin{equation}\label{eq:UVdeng}
\Delta(\mathbf{k} \cdot U) = \mathbf{k} \cdot \Delta(U) = \mathbf{k} \cdot \Delta(V) = \Delta(\mathbf{k} \cdot V).
\end{equation}
Since $\{-1, +1\}$ and $\mathrm{Sym}(m)$ are both finite, there exist 
$\mathbf{k}_1, \ldots, \mathbf{k}_d \in K_1$ such that 
$\mathrm{span}\{\mathbf{k}_1, \ldots, \mathbf{k}_d\} = \mathbb{R}^d$ and
\begin{equation}\label{eq:dengshi}
\langle \mathbf{k}_i, \mathbf{u}_j \rangle = s \cdot \langle \mathbf{k}_i, \mathbf{v}_{\pi(j)} \rangle + t_i,
\quad j = 1, \ldots, m,\ i = 1, \ldots, d,
\end{equation}
where $t_i \in \mathbb{Z}$, $s \in \{-1, +1\}$, and $\pi \in \mathrm{Sym}(m)$ are independent of $i$.
Applying \eqref{eq:dengshi} and the fact that $\{\mathbf{k}_1, \ldots, \mathbf{k}_d\}$ spans $\mathbb{R}^d$,
we conclude that
$$
\mathbf{u}_j = s \cdot \mathbf{v}_{\pi(j)} + \mathbf{z}_0, \quad j = 1, \ldots, m,
$$
for some $\mathbf{z}_0 \in \mathbb{R}^d$. Since $U, V \subset \mathbb{Z}^d$ and $s \in \{-1, +1\}$,
it follows that $\mathbf{z}_0 \in \mathbb{Z}^d$, which completes the proof of \eqref{eq:deng}.

\end{proof}
We are now ready to state the main result for the real domain $\mathbb{R}$ 
under the collision-free assumption.

\begin{thm}\label{thm: real_result}
Let $m \neq 6$ be a positive integer, and let $P := \{a_1, \ldots, a_m\} \subset \R$ satisfy the collision-free condition. If there exists a set $Q := \{b_1, \ldots, b_m\} \subset \R$ such that $Q - Q = P - P$ as multisets, then there exists some $z_0 \in \R$ such that
\begin{equation}\label{eq:QPz0}
Q = P + z_0 \quad \text{or} \quad Q = -P + z_0.
\end{equation}
\end{thm}

\begin{proof}
Let $P = \{a_1, \ldots, a_m\} \subset \mathbb{R}$ and $Q = \{b_1, \ldots, b_m\} \subset \mathbb{R}$ 
satisfy $Q - Q = P - P$ as multisets.
Consider the finite-dimensional $\mathbb{Q}$-vector space spanned by all elements of $P$ and $Q$:
$$
W = \mathrm{span}_{\mathbb{Q}}\{a_1, \ldots, a_m, b_1, \ldots, b_m\} \subset \mathbb{R}.
$$
Let $\{e_1, \ldots, e_d\} \subset \mathbb{R}$ be a $\mathbb{Q}$-basis of $W$. 
Every element of $P$ and $Q$ can be uniquely written as a $\mathbb{Q}$-linear combination of the basis vectors:

$$
a_i = \boldsymbol{\alpha}_i \cdot \mathbf{e}, \qquad b_i = \boldsymbol{\beta}_i \cdot \mathbf{e},
$$
where $\mathbf{e} := (e_1, \ldots, e_d) \in \mathbb{R}^d$, 
$\boldsymbol{\alpha}_i = (\alpha_{i,1}, \ldots, \alpha_{i,d}) \in \mathbb{Q}^d$, and 
$\boldsymbol{\beta}_i = (\beta_{i,1}, \ldots, \beta_{i,d}) \in \mathbb{Q}^d$.
Since $\{e_1, \ldots, e_d\}$ is a $\mathbb{Q}$-basis of $W$, the condition $Q - Q = P - P$ 
implies that 
$U := \{\boldsymbol{\alpha}_1, \ldots, \boldsymbol{\alpha}_m\} \subset \mathbb{Q}^d$ and 
${V} := \{\boldsymbol{\beta}_1, \ldots, \boldsymbol{\beta}_m\} \subset \mathbb{Q}^d$ 
satisfy the collision-free condition and ${V} - {V} = U - U$ as multisets.
By Lemma~\ref{lem:hig}, there exists 
$\mathbf{z}_0 = (z_{0,1}, \ldots, z_{0,d}) \in \mathbb{Q}^d$ such that
$$
\{\boldsymbol{\alpha}_1, \ldots, \boldsymbol{\alpha}_m\} 
= \{\boldsymbol{\beta}_1, \ldots, \boldsymbol{\beta}_m\} + \mathbf{z}_0
\quad \text{or} \quad
\{\boldsymbol{\alpha}_1, \ldots, \boldsymbol{\alpha}_m\} 
= -\{\boldsymbol{\beta}_1, \ldots, \boldsymbol{\beta}_m\} + \mathbf{z}_0.
$$
Translating back via the basis, this gives
$$
Q = P + z_0 \quad \text{or} \quad Q = -P + z_0,
$$
where $z_0 = \sum_{k} z_{0,k} \cdot e_k \in \mathbb{R}$, which completes the proof.
 
\end{proof}

\section{Proof of Theorem \ref{th:fanli}}
\label{sec: sec5}

%In this section, we demonstrate that, for any integer $m \geq 3$, there exist functions $f, g \in \mathcal{I}_m$ such that $\abs{\hat{f}(\omega)} = \abs{\hat{g}(\omega)}$ for all $\omega\in {\mathbb R}$, yet $f \not\sim g$. 
We begin with the following result, derived from Theorem~\ref{th:realfac} and Corollary \ref{co:ine}, 
which can be used to construct counterexamples to phase retrieval uniqueness, 
i.e., pairs $f \not\sim g$ such that
$$
|\hat{f}(\omega)| = |\hat{g}(\omega)| \quad \text{for all } \omega \in \mathbb{R}.
$$
\begin{thm}\label{th:fgxy}
Suppose that 
\[
f(x):=\sum_{k=1}^{m}\mathds{1}_{\big[\frac{a_k}{P},\frac{b_k}{P}\big]}(x)\in \mathcal{I}_m, \quad g(x):=\sum_{k=1}^{m'}\mathds{1}_{\big[\frac{a'_k}{P},\frac{b'_k}{P}\big]}(x)\in \mathcal{I}_{m'},
\]
where $a_k, b_k, a'_k, b'_k$ are nonnegative integers and $P$ is a positive integer. 
Define the polynomials
\begin{equation}\label{eq:XYZ}
X(z):=\sum_{k=1}^m z^{b_k}- \sum_{k=1}^m z^{a_k}, \qquad Y(z):=\sum_{k=1}^{m'} z^{b'_k}- \sum_{k=1}^{m'} z^{a'_k}
\end{equation}
for $z\in \mathbb{C}$.
Then $|\widehat{f}(\omega)|=|\widehat{g}(\omega)|$ for all $\omega \in \mathbb{R}$ if and only if 
there exists a real polynomial factorization $Y = R_1 \cdot R_2$, $\gamma\in \{1,-1\}$  and $\tau \in \mathbb{Z}$ such that
\begin{equation*}
X(z) = \gamma\cdot z^\tau \cdot R_1(z) \cdot {R_2(1/z)}.
\end{equation*}
{Furthermore, if $1\leq {\rm deg}(R_1)<{\rm deg}(Y)$,  $\gamma_1\cdot z^{\tau_1}\cdot  R_1(1/z)\not\equiv R_1(z)$ and $\gamma_2\cdot z^{\tau_2} \cdot R_2(1/z)\not\equiv R_2(z)$,  for any $\gamma_1,\gamma_2\in \{-1,+1\}$ and any integers $\tau_1$ and $\tau_2$, then $f\not\sim g$.}
\end{thm}

\begin{proof}
A simple calculation shows that 
  $\abs{\widehat{f}(\omega)}= \abs{\widehat{g}(\omega)}$ for all $\omega \in {\mathbb R}$ if and only if 
  \begin{equation}\label{eq:deng1}
  \abs{\sum_{k=1}^{m}{\exp(2\pi \mathrm{i}\cdot b_k\cdot \omega)-\exp(2\pi \mathrm{i} \cdot a_k\cdot \omega)}}
  =  \abs{\sum_{k=1}^{m'}{\exp(2\pi \mathrm{i}\cdot b'_k\cdot \omega)-\exp(2\pi \mathrm{i} \cdot a'_k\cdot \omega)}}.
  \end{equation}
  Set $d:=\max\{a_1,\ldots,a_m, b_1,\ldots,b_m,a'_1,\ldots,a'_{m'}, b'_1,\ldots,b_{m'} \}+1$ and
  \[
  {\mathbf x}:=(x_0,\ldots,x_{d-1})\in {\mathbb R}^d,\quad  {\mathbf y}:=(y_0,\ldots,y_{d-1})\in {\mathbb R}^d.
  \]
  Here,
  \begin{align*}
x_j &= \begin{cases}
-1, & \text{if } j \in \{a_1, \ldots, a_m\}, \\
1, & \text{if } j \in \{b_1, \ldots, b_m\},\\
0, & \text{else},
\end{cases}\quad 
y_j = \begin{cases}
-1, & \text{if } j \in \{a'_1, \ldots, a'_{m'}\}, \\
1, & \text{if } j \in \{b'_1, \ldots, b'_{m'}\},\\
0, & \text{else}.
\end{cases}
\end{align*}
A simple observation is that 
(\ref{eq:deng1}) holds if and only if $\abs{\hat{\mathbf x}(\omega)}=\abs{\hat{\mathbf y}(\omega)}$ for all $\omega\in {\mathbb R}$. 
Here, $\hat{\mathbf{x}}$ and $\hat{\mathbf{y}}$ denote the discrete Fourier transforms of $\mathbf{x}$ and $\mathbf{y}$, respectively, as defined in (\ref{eq:dftxy}).
   Then  $\abs{\widehat{f}(\omega)}= \abs{\widehat{g}(\omega)}$ for all $\omega \in {\mathbb R}$ if and only if    $\abs{\widehat{\mathbf x}(\omega)}= \abs{\widehat{\mathbf y}(\omega)}$ for all $\omega \in {\mathbb R}$. According to Theorem \ref{th:realfac}, we arrive at the if and only if conclusion. Furthermore,  it follows from Corollary \ref{co:ine}  that $f\not\sim g$ { if $1\le \deg(R_1)<\deg(Y)$, $\gamma_1\cdot z^{\tau_1}\cdot  R_1(1/z)\not\equiv R_1(z)$ and $\gamma_2\cdot z^{\tau_2} \cdot R_2(1/z)\not\equiv R_2(z)$,  for any $\gamma_1,\gamma_2\in \{-1,+1\}$ and any integers $\tau_1$ and $\tau_2$.}
  
\end{proof}

We next present the proof of Theorem \ref{th:fanli}
%\begin{thm}\label{th:fanli}
%For any integer $m\geq 3$, there exist functions  $f_0,g_0\in \cI_m$ such that $\abs{\widehat{f_0}(\omega)}=\abs{\widehat{g_0}(\omega)}$ for all $\omega\in %{\mathbb R}$ but $f_0\not\sim g_0$.
%\end{thm}
\begin{proof}[ Proof of Theorem \ref{th:fanli}]
For any integer $m \geq 3$, we construct functions 
\[
f_m(x)=\sum_{k=1}^{m}\mathds{1}_{[{a_k},{b_k}]}(x)\in \cI_m, \quad g_m(x)=\sum_{k=1}^{m}\mathds{1}_{[{a'_k},{b'_k}]}(x)\in \cI_m
\]
where the nonnegative integers $a_k, b_k, a'_k, b'_k$ are chosen to satisfy the following conditions. Let
\[
X_m(z):=\sum_{k=1}^m z^{b_k}- \sum_{k=1}^m z^{a_k}, \qquad Y_m(z):=\sum_{k=1}^{m} z^{b'_k}- \sum_{k=1}^{m} z^{a'_k}.
\]
According to Theorem  \ref{th:fgxy}, in order to construct $f_m$ and $g_m$ satifying $\abs{\widehat{f_m}(\omega)}=\abs{\widehat{g_m}(\omega)}$ for all $\omega \in {\mathbb R}$ and $f_m\not\sim g_m$, we require that there exist real coefficient polynomials $R_{m,1}, R_{m,2}$ with $Y_{m}(z) = R_{m,1}(z) \cdot R_{m,2}(z)$,
and there exists an correponding integer $\tau_m \in \mathbb{Z}$ such that
\[
X_m(z) =  z^\tau \cdot R_{m,1}(z) \cdot R_{m,2}(z^{-1}).
\]
{Here, $1\leq {\rm deg}(R_{m,1})<{\rm deg}(Y_m)$ $\gamma_1\cdot z^{\tau_1}\cdot  R_{m,1}(1/z)\not\equiv R_{m,1}(z)$ and $\gamma_2\cdot z^{\tau_2} \cdot R_{m,2}(1/z)\not\equiv R_{m,2}(z)$,  for any $\gamma_1,\gamma_2\in \{-1,+1\}$ and any integers $\tau_1$ and $\tau_2$.}

We construct counterexamples $f_m,g_m$ with correponding $R_{m,1}$ and $R_{m,2}$ for $m = 3, 4, 5, 6$, and more 
generally for $m = 3n+1$ with $n \geq 2$, $m = 3n+2$ with $n \geq 2$, 
and $m = 3n$ with $n \geq 3$, thereby covering all integers $m \geq 3$.

\textbf{(1) Case: $m = 3$}. 

Let $f_3, g_3: \mathbb{R} \rightarrow \mathbb{R}$ be defined as
\[
f_3(x):=\mathds{1}_{[0,5]}(x)+ \mathds{1}_{[6,9]}(x)+ \mathds{1}_{[11,12]}(x)\in \cI_3,\quad \text{and}\quad g_3(x):=\mathds{1}_{[0,1]}(x)+\mathds{1}_{[2,8]}(x)+\mathds{1}_{[10,12]}(x)\in \cI_3.
\]
The polynomials $R_{3,1}$   and $R_{3,2}$   associated with $f_3$  and $g_3$
  are given by
\[
\begin{aligned}
R_{3,1}(z):=(z-1)\cdot (z^7+z^2+1),\quad\text{and}\quad R_{3,2}(z):=z^4+z+1.
\end{aligned}
\]

\textbf{(2) Case: $m=4$}.

Let $f_4, g_4: \mathbb{R} \rightarrow \mathbb{R}$ be defined as
\[
\begin{aligned}
f_4(x):=&\mathds{1}_{[0,1]}(x) + \mathds{1}_{[4,7]}(x) + \mathds{1}_{[9,11]}(x) + \mathds{1}_{[12,14]}(x)\in \cI_4,\\
g_4(x):=&\mathds{1}_{[0,1]}(x) + \mathds{1}_{[3,5]}(x) + \mathds{1}_{[6,9]}(x) + \mathds{1}_{[12,14]}(x)\in \cI_4.
\end{aligned}
\]
The polynomials $R_{4,1}$   and $R_{4,2}$   associated with $f_4$  and $g_4$
  are given by
\[
\begin{aligned}
R_{4,1}(z):=(z - 1)\cdot (z^3 - z + 1), \ \text{and}\ 
R_{4,2}(z):=z^{10} + z^9 + z^8 + z^7 + z^6 + z^2 + z + 1.
\end{aligned}
\]

\textbf{(3) Case: $m=5$}. 

Let $f_5, g_5: \mathbb{R} \rightarrow \mathbb{R}$ be defined as
\[
\begin{aligned}
f_5(x):=&\mathds{1}_{[0,1]}(x) + \mathds{1}_{[3,4]}(x) + \mathds{1}_{[5,6]}(x) + \mathds{1}_{[7,8]}(x) + \mathds{1}_{[10,13]}(x)\in \mathcal{I}_5,\\
g_5(x):=&\mathds{1}_{[0,1]}(x) + \mathds{1}_{[2,3]}(x) + \mathds{1}_{[4,6]}(x) + \mathds{1}_{[7,8]}(x) + \mathds{1}_{[11,13]}(x)\in \mathcal{I}_5.
\end{aligned}
\]
 The polynomials $R_{5,1}$   and $R_{5,2}$   associated with $f_5$  and $g_5$
  are given by
\[
\begin{aligned}
R_{5,1}(z):=(z - 1)\cdot (z^7 + z^6 + z^5 + z^4 + z^3 + z^2 + 1), \quad
R_{5,2}(z):=z^5 - z^2 + 1.
\end{aligned}
\]

\textbf{(4) Case: $m=6$.}

Let $f_6, g_6: \mathbb{R} \rightarrow \mathbb{R}$ be defined as
\[
\begin{aligned}
f_6(x):=&\mathds{1}_{[0,1]}(x) + \mathds{1}_{[2,4]}(x) + \mathds{1}_{[6,7]}(x) + \mathds{1}_{[8,10]}(x) + \mathds{1}_{[11,12]}(x) + \mathds{1}_{[13,15]}(x)\in \mathcal{I}_6,\\
g_6(x):=&\mathds{1}_{[0,2]}(x) + \mathds{1}_{[3,4]}(x) + \mathds{1}_{[6,8]}(x) + \mathds{1}_{[9,10]}(x) + \mathds{1}_{[11,13]}(x) + \mathds{1}_{[14,15]}(x)\in \mathcal{I}_6.
\end{aligned}
\]
 The polynomials $R_{6,1}$   and $R_{6,2}$   associated with $f_6$  and $g_6$
  are given by
\[
\begin{aligned}
R_{6,1}(z):=(z - 1)\cdot (z^{11} + z^6 + 1), \quad
R_{6,2}(z):=z^3 + z^2 + 1.
\end{aligned}
\]

\textbf{(5) Case: $m = 3n + 1$ with $n \geq 2$.}

Let $f_m, g_m: \mathbb{R} \rightarrow \mathbb{R}$  be defined as
\[
\begin{aligned}
f_m(x):=&\mathds{1}_{[0,1]}(x)+\mathds{1}_{[3,6]}(x)+\mathds{1}_{[7,8]}(x)+\mathds{1}_{[9,10]}(x)\\
&\quad +\sum_{k=1}^{n-1} \mathds{1}_{[4+7k,4+7k+1]}(x)+\sum_{k=1}^{n-1} \mathds{1}_{[7+7k,7+7k+1]}(x)+\sum_{k=1}^{n-1} \mathds{1}_{[9+7k,9+7k+1]}(x)\in \mathcal{I}_m,\\
g_m(x):=
&\mathds{1}_{[0,1]}(x)+\mathds{1}_{[2,3]}(x)+\mathds{1}_{[4,7]}(x)+\mathds{1}_{[9,10]}(x)\\
&\quad +\sum_{k=1}^{n-1} \mathds{1}_{[4+7k,4+7k+1]}(x)+\sum_{k=1}^{n-1} \mathds{1}_{[7+7k,7+7k+1]}(x)+\sum_{k=1}^{n-1} \mathds{1}_{[9+7k,9+7k+1]}(x)\in \mathcal{I}_m\\
\end{aligned}
\]
The polynomials $R_{m,1}$   and $R_{m,2}$   associated with $f_m$  and $g_m$
  are given by
\[
\begin{aligned}
R_{m,1}(z):=(z - 1)\cdot \left(1+\sum_{k=0}^{n-1}z^{4+7k}\right), \quad
R_{m,2}(z):=z^5+z^3+1.
\end{aligned}
\]

\textbf{(6) Case: $m=3n+2$ with $n\geq 2$.}

 Let $f_m, g_m: \mathbb{R} \rightarrow \mathbb{R}$  be defined as
\[
\begin{aligned}
f_m(x)
:=&\mathds{1}_{[0,1]}(x)+\mathds{1}_{[3,4]}(x)+\mathds{1}_{[5,7]}(x)+\mathds{1}_{[9,10]}(x)+\mathds{1}_{[11,12]}(x),\\
&+\sum_{k=1}^{n-1} \mathds{1}_{[6+7k,6+7k+1]}(x)+\sum_{k=1}^{n-1} \mathds{1}_{[9+7k,9+7k+1]}(x)+\sum_{k=1}^{n-1} \mathds{1}_{[11+7k,11+7k+1]}(x)\in \mathcal{I}_m,\\
g_m(x)
:=
&\mathds{1}_{[0,1]}(x)+\mathds{1}_{[2,3]}(x)+\mathds{1}_{[5,7]}(x)+\mathds{1}_{[8,9]}(x)+\mathds{1}_{[11,12]}(x)\\
&+\sum_{k=1}^{n-1} \mathds{1}_{[6+7k,6+7k+1]}(x)+\sum_{k=1}^{n-1} \mathds{1}_{[8+7k,8+7k+1]}(x)+\sum_{k=1}^{n-1} \mathds{1}_{[11+7k,11+7k+1]}(x)\in \mathcal{I}_m.
\end{aligned}
\]
The polynomials $R_{m,1}$   and $R_{m,2}$   associated with $f_m$  and $g_m$
  are given by
\[
\begin{aligned}
R_{m,1}(z):=(z - 1)\cdot \left(1+\sum_{k=0}^{n-1}z^{6+7k}\right), \quad
R_{m,2}(z):=z^5+z^3+1.
\end{aligned}
\]

\textbf{(7) Case: $m=3n$ with $n\geq 3$}. 

Let $f_m, g_m: \mathbb{R} \rightarrow \mathbb{R}$  be defined as
\[
\begin{aligned}
f_m(x):=
&\mathds{1}_{[0,1]}(x)+\mathds{1}_{[3,4]}(x)+\mathds{1}_{[5,6]}(x)\\
&+\sum_{k=0}^{n-2} \mathds{1}_{[7+8k,7+8k+1]}(x)+\sum_{k=0}^{n-2} \mathds{1}_{[10+8k,10+8k+1]}(x)+\sum_{k=0}^{n-2} \mathds{1}_{[12+8k,12+8k+1]}(x)\in \mathcal{I}_m,\\
g_m(x):=
&\mathds{1}_{[0,1]}(x)+\mathds{1}_{[2,3]}(x)+\mathds{1}_{[5,6]}(x)\\
&+\sum_{k=0}^{n-2} \mathds{1}_{[7+8k,7+8k+1]}(x)+\sum_{k=0}^{n-2} \mathds{1}_{[9+8k,9+8k+1]}(x)+\sum_{k=0}^{n-2} \mathds{1}_{[12+8k,12+8k+1]}(x)\in \mathcal{I}_m.
\end{aligned}
\]
The polynomials $R_{m,1}$   and $R_{m,2}$   associated with $f_m$  and $g_m$
  are given by
\[
\begin{aligned}
R_{m,1}(z):=(z - 1)\cdot \left(1+\sum_{k=0}^{n-2}z^{7+8k}\right), \quad
R_{m,2}(z):=z^5+z^3+1.
\end{aligned}
\]

\end{proof}

\begin{comment}
\begin{example}
\label{example: m_3}
We now present counter-examples for $m=3,4,5,\ldots$. Specifically, for any given $m \in \{3,4,\ldots\}$, we construct functions $f$ and $g$ with $m$ intervals such that $|\widehat{f}|=|\widehat{g}|$, yet $f \not\sim g$. It should be emphasized that these counter-examples are not unique; we provide only one possible construction for each case. For each pair of counter-examples $f$ and $g$, to establish that $|\widehat{f}|=|\widehat{g}|$, we also present the corresponding functions $X_1, X_2$, and $X_3$ as described in Lemma \ref{lem: key} for convenience.
\end{example}
\end{comment}

\section{Separation property: a sufficient condition for Fourier phase retrieval}\label{sec:sep}
\label{sec: sec6}

As shown in Theorem \ref{th:fanli}, for any $m\geq 3$, there exists an $f_m\in \cI_m$
that cannot be uniquely determined from $\abs{\widehat{f}_m(\omega)}$ for all
$\omega\in \mathbb{R}$. The purpose of this section is to establish a sufficient
condition under which any $f\in \cI_m$ can be uniquely determined by
$\abs{\widehat{f}(\omega)}$ for all $\omega\in \mathbb{R}$.

This condition is based on the separation condition, originally introduced in
\cite{fienup}, which we recall as follows.

 \begin{defn}\label{de:sep}\cite{fienup}
 We say the $m$ intervals $I_1,\ldots,I_m\subset {\mathbb R}$ satisfy the {\em separation condition} if 
 \[
(I_{i_1}-I_{i_2})\cap (I_{j_1}-I_{j_2})=\emptyset,
\]
 for $1\leq i_1,i_2,j_1,j_2\leq m$, $j_1\neq j_2$, and $(i_1,i_2)\neq (j_1,j_2)$, where $(\ ,\ )$ denotes an ordered pair, and 
\[
I_{i_1}-I_{i_2}:=\{x-y\ :\ x\in I_{i_1},\ y\in I_{i_2}\}. 
\]
 \end{defn}

%\begin{rem}
{Intuitively, the separation condition requires the intervals to be 
sufficiently far apart relative to their lengths. More precisely, 
given $m$ intervals $I_j = [a_j, b_j]$, $j = 1, \ldots, m$, with
\[
a_1 < b_1 < a_2 < b_2 < \cdots < a_m < b_m,
\]
the separation condition requires that
\[
\min_{1 \leq j \leq m-1}(a_{j+1} - b_j) > \max_{1 \leq j \leq m}(b_j - a_j).
\]}

\begin{comment}
%when adding new interval $I_{m+1}=[a_{m+1},b_{m+1}]$ such that $b_m<a_{m+1}<b_{m+1}$. Then the $m+1$ intervals $I_j$, $j=1,\ldots,m+1$, satisfy the separation condition if and only if
%\[
%a_{m+1}-b_m>\max\{b_m-a_1,\ b_{m+1}-a_{m+1}\}.
%\]
%In other words, the distance from $I_{m+1}$ to the convex hull of $\bigcup_{j=1}^m I_j$ must be larger than both the length of this convex hull and the length of $I_{m+1}$. 
%The conclusion is not correct! 例如取I_1=[0,0+\epsilon], I_1=[3,3+\epsilon], I_2=[3,3+\epsilon], (\epsilon为一个很小的正数)
%\end{rem}
\end{comment}
The following theorem shows that if $f \in \mathcal{I}_m$ whose 
$m$ intervals satisfy the separation condition, then $f$ 
can be uniquely determined by $|\widehat{f}(\omega)|$ for all 
$\omega \in \mathbb{R}$.
 
\begin{thm}
\label{thm: separation_condition}
Assume that the function $f\in \cI_m$ is denoted as $f(x)=\sum_{j=1}^{m} \mathds{1}_{I_j}(x)$, where $m\neq 6$.  The intervals $I_{j}:=[a_j,b_j], j=1,\ldots,m$ with $a_k,b_k\in {\mathbb{R}}$ and $a_1<b_1<\ldots<a_m<b_m$,  satisfy the  {separation condition}
Then, $f$ can be uniquely determined by $|\widehat{f}(\omega)|$ for any $\omega \in {\mathbb R}$. 
\end{thm}
 \begin{rem}
Here we require $m \neq 6$. One may naturally wonder whether the conclusion still holds when $m = 6$. Unfortunately, this is not the case: when $m = 6$, the separation condition becomes insufficient to uniquely determine the indicator function $f$.
 For example, consider
\[
\begin{aligned}
f(x)&=\mathds{1}_{[0,0.25]}(x)+\mathds{1}_{[1,1.25]}(x)+\mathds{1}_{[4,4.25]}(x)+\mathds{1}_{[10,10.25]}(x)+\mathds{1}_{[12,12.25]}(x)+\mathds{1}_{[17,17.25]}(x),\\
g(x)&=\mathds{1}_{[0,0.25]}(x)+\mathds{1}_{[1,1.25]}(x)+\mathds{1}_{[8,8.25]}(x)+\mathds{1}_{[11,11.25]}(x)+\mathds{1}_{[13,13.25]}(x)+\mathds{1}_{[17,17.25]}(x).
\end{aligned}
\]
By direct calculation, both functions involve exactly 6 intervals that satisfy the separation condition. We observe that $|\widehat{f}(w)|=|\widehat{g}(w)|$
 for all $w\in \mathbb{R}$, yet $f \not\sim g$, demonstrating that the separation condition alone cannot guarantee unique determination when $m=6$.
\end{rem}

\begin{comment}
\begin{rem}
The lemma above can be readily extended to the rational domain $\mathbb{Q}$ by multiplying all elements in $X$ and $Y$ by a sufficiently large integer to convert them into integers. We can then apply the above result by replacing $\mathbb{Z}$ with $\mathbb{Q}$, without significantly modifying the details of the proof.
\end{rem}
\end{comment}
We next present the proof of  Theorem \ref{thm: separation_condition}.  
\begin{proof}[Proof of  Theorem \ref{thm: separation_condition}]

Let $F:\mathbb{R}\rightarrow \mathbb{R}$ be defined by $F(x)=(f\mycircledast f_-)(x)$, where $f_-(x)=f(-x)$. Since 
\[
\widehat{F}(\omega)=|\widehat{f}(\omega)|^2,
\]
determining $f$ from $|\widehat{f}|$ is equivalent to showing that $f$ can be determined by $F$. The following proof is divided into three steps.

\textbf{Step 1: Determine the number of intervals $m$ of $f$.}

Based on the definition of $F$, we observe that $F(x)=F(-x)$ holds for any $x\in \mathbb{R}$, indicating that $F$ is an even function. Consequently,  the support of $F$, defined as 
\[
{\text{supp}(F)}:=\overline{\{x\ :\ F(x)\neq 0\}},
\]
 exhibits a symmetric structure and can be expressed as the union of disjoint intervals:
\begin{equation}
\label{eqn: interval_expression}
{\text{supp}(F)}=\left(\cup_{k=1}^{M}[e_k,f_k]\right)\cup  \left(\cup_{k=1}^{M}[-f_k,-e_k]\right)\cup [-e_0,e_0],
\end{equation}
with some $M\in \mathbb{Z}_{+}$ and $0<e_0<e_1<f_1<\cdots<e_M<f_M$.
Through direct calculation, we obtain:
\[
\left(f\mycircledast f_-\right)(x)=\sum_{1\leq k_1,k_2\leq m} \left(\mathds{1}_{I_{k_1}}\mycircledast \mathds{1}_{-I_{k_2}}\right)(x)
\]
Using (\ref{eqn: ab_cd}), we determine that the support of each convolution term is given by:
\[
\text{supp}(\mathds{1}_{I_{k_1}}\mycircledast \mathds{1}_{-I_{k_2}})=I_{k_1}-I_{k_2}. 
\]
The separation condition ensures that the supports of different convolution terms are disjoint, i.e.,
\begin{equation}
\label{eqn: supp}
{\text{supp}(\mathds{1}_{I_{j_1}}\mycircledast \mathds{1}_{-I_{j_2}})}\cap {\text{supp}(\mathds{1}_{I_{i_1}}\mycircledast \mathds{1}_{-I_{i_2}})}=\emptyset, 
\end{equation}
for any $1\leq j_1,j_2,i_1,i_2\leq m$, $i_1\neq i_2$, and $(j_1,j_2)\neq (i_1,i_2)$. 
Additionally, we note that the diagonal terms $(j,j)$, $j=1,\ldots,m$, contribute to the central interval:
\[
\cup_{j=1}^{m}{\text{supp}(\mathds{1}_{I_{j}}\mycircledast \mathds{1}_{-I_{j}})}=[-e_0,e_0],
\]

Since there are precisely $m(m-1)/2$ distinct pairs $(i_1,i_2)$
  with $1\leq i_1<i_2\leq m$, we establish the relationship between $M$ and $m$ as:
\[
M\,\,=\,\,\frac{m(m-1)}{2}.
\]
Therefore, when $F$ is given with its support  expressed as in (\ref{eqn: interval_expression}) containing $2M+1$ intervals, the number $m$ of intervals for $f$ can be uniquely determined by solving the quadratic equation above, yielding:
\[
m=\left\lfloor \sqrt{2M} \right\rfloor + 1.
\]

\textbf{Step 2: Determine the midpoint of each interval $I_j=[a_j,b_j]$ for $j=1,\ldots,m$.}

Let $c_j$ denote the midpoint of each interval $[a_j,b_j]$, that is, $c_j:=(a_j+b_j)/2$ for $j=1,\ldots,m$. Similarly, let $d_j$ denote the midpoint of each interval $[e_j,f_j]$ given in (\ref{eqn: interval_expression}), that is, $d_j:=(e_j+f_j)/2$ for $j=1,\ldots,M$, and set
$d_{-j}:=-d_j, j=1,\ldots,M$, 
 $d_0:=0$. Define the set $X:=\{c_1,\ldots,c_m\}$. Then, by direct calculation, we have
$X-X=D$, where  
\[
D=\{d_{-m(m-1)/2},d_{-m(m-1)/2+1},\ldots,d_{-1},d_{0},d_{1},\ldots,d_{m(m-1)/2-1},d_{m(m-1)/2}\}.
\]
Moreover, the separation condition ensures that all elements of $D$ 
are distinct. Let $c_0$ be a positive integer such that all elements 
of $c_0 \cdot X - c_0 \cdot X$ and $c_0 \cdot D$ are integers. By 
Lemma~\ref{thm: real_result}, for $m \neq 6$, the set 
$c_0 X = \{c_0\cdot c_j\}_{j=1}^m$ is uniquely determined from 
$c_0 X - c_0 X = c_0 D$ up to shift and reflection. Consequently, 
$X$ is also uniquely determined from $X - X$. Without loss of 
generality, write $\{c_j\}_{j=1}^m$ with $c_1 < c_2 < \cdots < c_m$.
  To eliminate the effects of shift and reflection, we set $c_1=0$
  and impose the condition $c_2-c_1<c_{m}-c_{m-1}$. The inequality $c_{2}-c_{1}\neq c_{m}-c_{m-1}$
is guaranteed by the separation condition, which states that:
\[
\text{supp}(\mathds{1}_{I_1}\mycircledast\mathds{1}_{-I_2})\cap \text{supp}(\mathds{1}_{I_{m-1}}\mycircledast\mathds{1}_{-I_m})=\emptyset.
\]

 \textbf{Step 3: Determination of the intervals $[a_j,b_j]$, for $j=1,\ldots,m$.}

Since $c_1, \ldots, c_m$ have already been determined, for each 
$j = 2, \ldots, m$ we can identify $d_{l_j}$ such that 
$c_j - c_1 = d_{l_j}$, thereby determining the midpoint $d_{l_j}$ 
of the interval $[e_{l_j}, f_{l_j}]$, and hence the interval 
$[e_{l_j}, f_{l_j}]$ itself. Considering the function $f(x) * f(-x)$ 
on $[e_{l_j}, f_{l_j}]$ and applying the formula in 
\eqref{eqn: ab_cd}, the corresponding set of interval lengths
\[
U_{j} := \{|I_j|, |I_1|\}
\]
is determined for each $j = 2, \ldots, m$. Taking the intersection 
$\bigcap_{j=2}^{m} U_j$ yields $|I_1|$, and consequently $|I_j|$ 
is determined for each $j = 2, \ldots, m$. Since both the midpoint 
$c_j$ and the interval length $|I_j|$ are now known for each 
$j = 1, \ldots, m$, the intervals $[a_j, b_j]$, $j = 1, \ldots, m$, 
are uniquely determined. This completes the proof.

\begin{comment}
Besides, based on the separation condition, we can see denote 
\[
\alpha=\min\{\min_{k} (e_{k+1}-f_{k}),\min_k (b_k-a_k)/2\}.
\]
Then $d_{k+1}-d_{k}\geq \alpha$.  Take $c_{k,R}\in \mathbb{Q}$, $k=1,\ldots,N$, such that 
\[
|c_{k,R}-c_k|\leq \alpha/6.
\]
 Then we can directly see that:
 \[
| c_{k_1,R}-c_{k_2,R}-(c_{k_1}-c_{k_2})|\leq \alpha/3.
 \]

Now denote $D_k=\{r\in \mathbb{Q}\ :\ |r-d_k|\leq \alpha/3\}$, we can directly see that the sets $D_{k}$, $k=1,\ldots,N(N-1)/2$ are disjoint. Take $d_{k,R}\in D_{k}$ such that $\{d_{k,R}\ :\ k=1,\ldots,N(N-1)/2\}$ is the different set of some set $\{s_1,\ldots,s_N\}$ with $s_1<s_2<\cdots<s_N$, then such kind of set is not empty as the statement in above. Based on Lemma XX???, we can see that $\{s_1,\dots,s_{N}\}$ is unique up to flip and shift ambiguity. 

Moreover, by the triangle inequality, we can see that:
\[
|s_k-c_k|\leq \alpha
\]
Therefore, we can see that $s_k\in [a_k,b_k]$. Since $s_{k_1}-s_{k_2}\in I_{k_1}-I_{k_2}$ with $k_1>k_2$, and the lengths $\{|I_{k_1}|,|I_{k_2}|\}$ can be determined by $F$, then the corresponding $I_{k_1}$, $I_{k_2}$ can be determined. Combined with the determination of $c_k$, all the informations of $f$ is now determined. 
\end{comment}
\end{proof}

\end{document}